\documentclass[a4paper,onecolumn,oneside,10pt]{article}

\usepackage{geometry}
\usepackage[x11names,usenames,dvipsnames]{xcolor} 
\usepackage{tikz}
\usepackage[T1]{fontenc}
\usepackage{comment}
\usepackage{lineno, amsthm}    
\usepackage{mathtools}
\usepackage{mathdots}
\usepackage{enumerate}
\usepackage{textcomp}
\usepackage{amssymb,amsfonts}
\usepackage{amsmath,paralist,enumitem,mathrsfs,graphicx}
\usepackage{epstopdf}
\usepackage{pgfplots}

\usetikzlibrary{pgfplots.dateplot} 
\renewcommand{\qedsymbol}{$\blacksquare$}
\newcommand\restr[2]{{
		\left.\kern-\nulldelimiterspace 
		#1 
		\vphantom{\big|} 
		\right|_{#2} 
}}
\makeatletter
\newcommand{\vas}{\bBigg@{3}}
\newcommand{\vast}{\bBigg@{4}}
\newcommand{\Vast}{\bBigg@{5}}
\makeatother
\newcommand{\upperRomannumeral}[1]{\uppercase\expandafter{\romannumeral#1}}
\newcommand{\lowerRomannumeral}[1]{\lowercase\expandafter{\romannumeral#1}}
\usepackage{indentfirst}
\usepackage[unicode]{hyperref}
\usepackage{appendix}
\newcommand{\dd}{\mathrm{d}}

\usepackage{epstopdf}
\usepackage[caption=false]{subfig}
\theoremstyle{plain}
\newtheorem{theorem}{Theorem}
\newtheorem{hyp}{Hypothesis}
\newtheorem{lemma}[theorem]{Lemma}
\newtheorem{corollary}[theorem]{Corollary}
\newtheorem{proposition}[theorem]{Proposition}

\theoremstyle{definition}
\newtheorem{rem}{Remark}

\modulolinenumbers[2]
\newenvironment{myproof}[1] {{
		
		\noindent\it{Proof of {#1}}. }}{\hfill\qedsymbol}
\makeatletter
\def\ps@pprintTitle{%
	\let\@oddhead\@empty
	\let\@evenhead\@empty
	\def\@oddfoot{\footnotesize\itshape
		\ifx\@empty\@empty
		\else\@journal\fi\hfill\today}%
	\let\@evenfoot\@oddfoot
}
\makeatother

\makeatletter
\let\save@mathaccent\mathaccent
\newcommand*\if@single[3]{%
	\setbox0\hbox{${\mathaccent"0362{#1}}^H$}%
	\setbox2\hbox{${\mathaccent"0362{\kern0pt#1}}^H$}%
	\ifdim\ht0=\ht2 #3\else #2\fi
}
\newcommand*\rel@kern[1]{\kern#1\dimexpr\macc@kerna}
\newcommand*\widebar[1]{\@ifnextchar^{{\wide@bar{#1}{0}}}{\wide@bar{#1}{1}}}
\newcommand*\wide@bar[2]{\if@single{#1}{\wide@bar@{#1}{#2}{1}}{\wide@bar@{#1}{#2}{2}}}
\newcommand*\wide@bar@[3]{%
	\begingroup
	\def\mathaccent##1##2{%
		\let\mathaccent\save@mathaccent
		\if#32 \let\macc@nucleus\first@char \fi
		\setbox\z@\hbox{$\macc@style{\macc@nucleus}_{}$}%
		\setbox\tw@\hbox{$\macc@style{\macc@nucleus}{}_{}$}%
		\dimen@\wd\tw@
		\advance\dimen@-\wd\z@
		\divide\dimen@ 3
		\@tempdima\wd\tw@
		\advance\@tempdima-\scriptspace
		\divide\@tempdima 10
		\advance\dimen@-\@tempdima
		\ifdim\dimen@>\z@ \dimen@0pt\fi
		\rel@kern{0.6}\kern-\dimen@
		\if#31
		\overline{\rel@kern{-0.6}\kern\dimen@\macc@nucleus\rel@kern{0.4}\kern\dimen@}%
		\advance\dimen@0.4\dimexpr\macc@kerna
		\let\final@kern#2%
		\ifdim\dimen@<\z@ \let\final@kern1\fi
		\if\final@kern1 \kern-\dimen@\fi
		\else
		\overline{\rel@kern{-0.6}\kern\dimen@#1}%
		\fi
	}%
	\macc@depth\@ne
	\let\math@bgroup\@empty \let\math@egroup\macc@set@skewchar
	\mathsurround\z@ \frozen@everymath{\mathgroup\macc@group\relax}%
	\macc@set@skewchar\relax
	\let\mathaccentV\macc@nested@a
	\if#31
	\macc@nested@a\relax111{#1}%
	\else
	\def\gobble@till@marker##1\endmarker{}%
	\futurelet\first@char\gobble@till@marker#1\endmarker
	\ifcat\noexpand\first@char A\else
	\def\first@char{}%
	\fi
	\macc@nested@a\relax111{\first@char}%
	\fi
	\endgroup
}

\title{Feller's test for explosions of stochastic Volterra equations\footnote{The first author acknowledges financial support from the research program \emph{Chaire Statistics and models for Regulation} (Fondation de l'\'Ecole polytechnique and \'Ecole polytechnique). The work of the second author benefited from the financial support of the research program \emph{Chaire Deep learning in finance and Statistics} (Fondation de l'\'Ecole polytechnique, \'Ecole polytechnique and Qube R\&T).}}

\author{Alessandro Bondi\thanks{Centre de Mathématiques Appliquées (CMAP), \'Ecole polytechnique, Institut Polytechnique de Paris.\\ Email:  alessandro.bondi@polytechnique.edu}\and Sergio Pulido\thanks{Universit\'e Paris-Saclay, CNRS, ENSIIE, Univ \'Evry, Laboratoire de Math\'ematiques et Mod\'elisation d'\'Evry (LaMME).\\ Email: sergio.pulidonino@ensiie.fr}}

\pgfplotsset{compat=1.18}
\begin{document}
	\maketitle
\begin{abstract}
	This paper provides a Feller's test for explosions of one-dimensional continuous stochastic Volterra processes of convolution type. The study focuses on dynamics governed by nonsingular kernels, which preserve the semimartingale property of the processes and introduce memory features through a path-dependent drift. In contrast to the classical path-independent case, the sufficient condition derived in this study for a Volterra process to remain in the interior of an interval is generally more restrictive than the necessary condition.	
	The results are illustrated with three specifications of the dynamics:  the Volterra square-root diffusion, the Volterra Jacobi process and the  Volterra power-type diffusion. For the Volterra square-root diffusion, also known as the Volterra CIR process, the paper presents a detailed discussion on the approximation of the singular fractional kernel with a sum of exponentials, a method commonly employed in the mathematical finance literature.
	\vspace{2mm}
	\\
	{\bf MSC2020:} 60H20; 45D05; 60K50
\vspace{1mm}\\	
	{\bf Keywords:} stochastic Volterra equations; boundary attainment conditions; stochastic invariance; Feller's test for explosions

\end{abstract}	
	\section{Introduction}\label{sec_Intro}
	Consider a stochastic basis $(\Omega,\mathcal{F},\mathbb{P};\mathbb{F}=(\mathcal{F}_t)_{t\ge0}),$ where the filtration $\mathbb{F}$ satisfies the usual conditions, endowed with a standard one-dimensional $\mathbb{F}-$Brownian motion $W=(W_t)_{t\ge0}$. Given two (possibly infinite) numbers $l,\,r$ such that $-\infty\le  l < r\le \infty $, define an open interval $I=(l,r)$ in the extended real line $\widebar{\mathbb R}$. Suppose that  a continuous process $X=(X_t)_{t\ge 0}$, starting at a deterministic point $X_0=x_0\in I$ and  taking values in the closed interval $\widebar {I}$, has the following property: for every $\lambda,\rho \in I$ such that $x_0\in (\lambda, \rho)$, defining $\tau_\lambda^\rho=\inf\{t\ge0 : X_t=\lambda \text{ or } X_t=\rho\}$, the stopped process $X^{\tau_\lambda^\rho}={(X_{t\wedge \tau_\lambda^\rho})}_{t\ge 0}$ solves the one-dimensional stochastic Volterra equation 
	\begin{equation}\label{eq:1intro}
		X_{t\wedge \tau_\lambda^\rho}=x_0+\int_{0}^{t}1_{\{s\le \tau_\lambda^\rho\}}K(t-s)b(X_{s\wedge \tau_\lambda^\rho})\,\dd s+ \int_{0}^{t}1_{\{s\le \tau_\lambda^\rho\}}K(t-s)\sigma(X_{s\wedge \tau_\lambda^\rho})\,\dd W_s,
		\quad t\ge 0, \,\mathbb{P}-\text{a.s.}
	\end{equation}
Here, $K\colon \mathbb{R}_+\to \mathbb{R}$ is a smooth  kernel ($K\in C^1(\mathbb{R}_+)$) which introduces  memory features in the equation.   
Under additional hypotheses on $K$, the objective of this paper is to determine boundary attainment conditions for the process $X$. More specifically, defining the $\mathbb{F}-$stopping time $S$ by  $$S=\inf\{t\ge 0 : X_t\in \{l,r\}\},$$ the aim is to establish a test based on the initial point $x_0$ and the parameters of the dynamics \eqref{eq:1intro} -- that is,  $K,\,b$ and $\sigma$ -- to deduce whether $\mathbb{P}(S=\infty)=1$ or not.  Notice that, by the continuity of $X$, $S$ represents the first time that the trajectories of $X$ attain $l$ or $r$. In case $l=-\infty$ and $r=\infty$, $S$ is typically called the first blow-up  (or explosion) time  of  $X$.\\  In other words, given a continuous $\widebar I-$valued process $X$ which is a local solution in the random interval $[0,S)$ (in the sense of \eqref{eq:1intro}) of the stochastic Volterra equation 
	\begin{equation*}
	X_{t}=x_0+\int_{0}^{t}K(t-s)b(X_s)\,\dd s+ \int_{0}^{t}K(t-s)\sigma(X_s)\,\dd W_s,
	\quad t\in [0,S), \,\mathbb{P}-\text{a.s.},
\end{equation*}
this work proposes to study necessary and sufficient conditions for $X$ to be also a global, $I-$valued solution. 

\vspace{2mm}

In the classical path-independent case $K\equiv K(0)$, when \eqref{eq:1intro} reduces to an It\^o diffusion, this problem has long been solved thanks to  \cite{WF}, with a strategy which has become known as \emph{Feller's test for explosions}, see also \cite[Theorem 5.29, Chapter 5]{KS}. In this approach, for a generic $c\in I$, the model parameters $K(0),\,b$ and $\sigma$ are used to define an auxiliary function $v_c\colon I \to \mathbb{R_+}$ which is then employed to establish the following equivalence:
\begin{equation}\label{eq:intro_cla}
	\mathbb{P}(S=\infty) \quad \Longleftrightarrow\quad v_c(l+)=v_c(r-)=\infty.
\end{equation}
Perhaps the most famous application of this result is the so-called \emph{Feller condition} for square-root diffusions, which have been extensively used, for instance, in mathematical finance to model interest rates \cite{CIR} and asset variance \cite{HE}. \vspace{2mm}
\\
In order to treat the case of a nonconstant convolution kernel $K\in C^1(\mathbb{R}_+)$, this paper generalizes the procedure carried out in the standard It\^o framework, making it feasible to analyze the path-dependence features of  \eqref{eq:1intro}. The foundation of the arguments developed here is the fact that, since $K(0)<\infty$, the stopped process $X^{\tau_\lambda^\rho}$ is a continuous semimartingale. In fact, it is possible to express its dynamics \eqref{eq:1intro} in a differential form where the It\^o diffusive component is separated from the path-dependent one. The path-dependent component is given in terms of the measure resolvent of the first kind $L$ of $K$, which solves the convolution identity $K\ast L=1$. This formulation allows the application of It\^o's formula, an essential tool for the proposed methodology. \\
The introduction, for every $\beta,\,\gamma\in\mathbb{R}$, of a modified drift function $\tilde{b}_c(\cdot;\beta,\gamma)\colon I\to \mathbb{R}$ defined by 
	\[
	\tilde{b}_c(x;\beta,\gamma)=K(0)b(x)+\frac{K'(0)}{K(0)}\left(x+1_{\{x<c\}}\beta+1_{\{x\ge  c\}}\gamma\right),\quad x\in I,\text{ for some }c\in I,
	\]
can be considered as the main idea of this manuscript. Indeed, roughly speaking, a suitable choice of the  shift parameters $\beta $ and $\gamma$, along with monotonicity assumptions on the convolution $K'\ast L$, enables to control the sign of the path-dependent component of \eqref{eq:1intro}, which turns out to be crucial for the analysis of $\mathbb{P}(S=\infty)$. Note that, in the path-independent case when $K'(0)=0$, $\tilde{b}_c(\cdot;\beta,\gamma)$ reduces to the drift map $b$ because the dependence on $\beta$, $\gamma$ and $c$ is lost.\\
 Employing $\tilde{b}_c(\cdot;\beta,\gamma)$, by analogy with the classical case, an auxiliary function $v_c(\cdot;\beta,\gamma)\colon I\to\mathbb{R}_+$ can be defined to study necessary and sufficient conditions for $\mathbb{P}(S=\infty)$. More precisely, given $c\in I$, Theorem \ref{Feller_thm_nec} shows that 
\begin{equation}\label{intro_nec}
	v_c(l+;-\beta,-\gamma)=v_c(r-;-\beta,-\gamma)=\infty, \text{ for every } \gamma<x_0<\beta,
\end{equation}
is a necessary condition for $\mathbb{P}(S=\infty)$. Meanwhile, Theorem \ref{Feller_thm}   asserts that, given two real-valued  sequences $(\lambda_n)_n$ and $(\rho_n)_n$ such that $\lambda_n\downarrow l$ and $\rho_n\uparrow r$ as $n\to \infty$,  it is sufficient that
\begin{equation}\label{intro_suf}
\lim_{n\to \infty}v_c(\lambda_n; -\lambda_n, -\rho_n)
=
\lim_{n\to\infty}v_c(\rho_n; -\lambda_n, -\rho_n)=\infty	
\end{equation}
to infer $\mathbb{P}(S=\infty)$. Contrary to the equivalence in \eqref{eq:intro_cla} which holds when $K\equiv K(0)$, in general, the requirement in \eqref{intro_suf} is stronger than the one in  \eqref{intro_nec}, see Lemma \ref{le_dip_betagamma}. However, the conditions \eqref{intro_nec}-\eqref{intro_suf} coincide when $I$ is bounded under further assumptions on $\sigma$, as detailed in Corollary \ref{coro_nec_suff}. These additional assumptions allow to choose $\beta$ and $\gamma$ equal to zero, bringing the setting back to the standard Feller's test formulation.

\vspace{2mm}
Stochastic Volterra equations appear in various fields of applied mathematics. For example, in biological and chemical models describing the interactions of two substances (called reactant and catalyst), they arise from and are motivated by scaling limits of branching processes, see \cite{ABI_j, MS}. 
They are also employed in volatility modeling to incorporate different levels of regularity of the trajectories and path-dependence structures, see, e.g., \cite{rhh, rh2} and references therein. Moreover, Volterra-type processes have been used in \cite{BBV1} for the stochastic modeling of energy markets,  see also \cite{BBV2}.
\\
The boundary attainment conditions obtained in this work contribute to expanding the theory of invariance for stochastic Volterra equations. For some time, progress in this research area was limited to \cite[Theorem 3.6]{sergio}, which, in a suitable setting,  proves the existence of $\mathbb{R}^d_+-$valued  solutions and  can be applied, e.g., to study Volterra square-root diffusion dynamics (see also \cite{rh2}). Recently, however, such invariance theory has garnered renewed interest. Notably, under Lipschitz continuity assumptions on the drift and diffusion maps, \cite[Theorem 3.2]{AKsolo}  shows the invariance of closed convex subsets of $\mathbb{R}^d$ for nonsingular kernels $K$ that preserve nonnegativity. This result also covers singular completely monotone kernels. 
Furthermore, \cite[Theorem 5.1]{Pol_sergio} establishes the existence of  weak continuous solutions to stochastic Volterra equations with singular kernels taking values in the multidimensional unit ball.  Here,  weaker regularity assumptions on the coefficients -- namely, continuity and linear growth -- are sufficient. This result can be applied to certain polynomial Volterra processes, including the Volterra Jacobi process, which is also examined in this paper in Subsection~\ref{subsection_JV}. The present article is the first study of invariance properties of stochastic Volterra equations over open sets. \\
It is worth noticing that the arguments of this manuscript focus on one-dimensional dynamics and require  the crucial hypothesis that $K(0)<\infty$, i.e., that the kernel $K$ does not blow up at $0$.  When $K(0+)=\infty$,  it is unclear whether the $I-$boundary attainment of a local solution $X$ to the corresponding stochastic Volterra equation (which is not a semimartingale anymore) can be inferred using \eqref{intro_nec}-\eqref{intro_suf} via stability techniques. Such techniques, employed for instance in \cite[Theorem 3.2]{AKsolo} to handle a singular completely monotone kernel, consist of  approximating  $K$ with a sequence of nonexplosive kernels $K^{(n)}$ (e.g., in the  $L^2_{\text{loc}}(\mathbb{R}_+)-$sense) such that the associated solutions $X^{(n)}$ converge to $X$ in law,  in the spirit of \cite[Theorem 3.5]{ABE1}.  In fact, for this strategy to work, given that the domain of interest is an open set, additional information might be needed on $\inf_{0\le t< S}X_t^{(n)}$ or $\sup_{0\le t< S}X_t^{(n)}$, as well as  on the distribution and convergence of the sequence of first boundary hitting times of $(X^{(n)})_n$. However, such results are not provided here or in the existing literature and might be too ambitious to achieve, rendering the stability approach unfeasible. Focusing  on explosions to $\pm\infty$ of local maximal solutions (whose existence is analyzed, for example, in \cite[Subsection 3.3]{Zhang}), when $K$ is the fractional kernel $K(t)=\Gamma(\alpha)^{-1}t^{\alpha-1}$ with $\alpha\in(0,1)$, an Osgood-type criterion for stochastic Volterra equations driven by additive (that is, $\sigma=1$) Riemann-Liouville  fractional Brownian motions  has been established in \cite{BPT}. On the other hand, the case of a more general multiplicative noise appears to be completely open and could be the scope of a future research. 

\vspace{2mm}
The paper is organized as follows. Section \ref{main} sets the stage by imposing suitable assumptions on the kernel $K$ and the coefficients $b$ and $\sigma$, and by deducing a differential expression for the dynamics \eqref{eq:1intro} of the stopped process $X^{\tau_\lambda^\rho}$, see Lemma \ref{le_semi}. Moreover, by analogy with the classical path-independent case, it introduces a so-called scale function $p_c(\cdot; \beta,\gamma)\colon I\to \mathbb{R}$ and the map $v_c(\cdot; \beta,\gamma)$ discussed above. Then, Subsection~\ref{suff_scalep} shows how to employ $p_c(\cdot;\beta,\gamma)$ to derive information on the $\inf$ and the $\sup$ of $X$ in the random interval $[0,S)$. Subsection~\ref{subsec_v} contains the main contribution of this work, which is a Feller-type test for the $I-$boundary attainment of $X$. This test is divided  into a necessary condition and a sufficient one for $\mathbb{P}(S=\infty)$, which, in general, do not coincide. Such conditions are presented in Theorems~\ref{Feller_thm_nec}-\ref{Feller_thm}. Section \ref{sec_ex} applies the theoretical results of Section \ref{main} to three specifications of \eqref{eq:1intro} which are frequent in the literature. In order of appearance, they are: the Volterra CIR process, the Volterra Jacobi process and the  Volterra power-type diffusion process. Each example is investigated in a separate subsection. In particular, Subsection~\ref{subsection_CIR} treats  Volterra square-root diffusions and the approximation of the fractional kernel with nonsingular kernels, a common practice in  mathematical finance.
Finally, Appendix \ref{Appendix_tech} contains the proofs of some technical results about the dependence of $p_c(\cdot;\beta,\gamma)$ and $v_c(\cdot;\beta,\gamma)$ on  $\beta,\,\gamma$ and $c$.

	\section{Boundary attainment conditions for stochastic Volterra equations}\label{main}
	We let $\mathbb{R}_+=[0,\infty)$ and denote by $\delta_a(\dd t)$ the Dirac measure on $\mathbb{R}_+$ centered in $a\in \mathbb{R}_+$. For a right-continuous function $f\colon \mathbb{R}_+\to \mathbb{R}$ of locally bounded variation, we let $\dd f(t)$ be the  Lebesgue-Stieltjes measure on the $\delta-$ring  $\cup_{n\in \mathbb{N}}\mathcal{B}([0,n])$ associated with $f$. By convention, $\dd f(\{0\})=0$.
	
	In this paper, we consider a kernel $K\colon\mathbb{R}_+\to\mathbb{R}_+$ satisfying the following requirement.
	\begin{hyp}\label{Hyp1}
		The kernel $K$ belongs to $C^1(\mathbb{R}_+)$, with $K'\in {\emph{BV}_\emph{loc}}(\mathbb{R}_+),$ and admits a nonnegative measure resolvent of the first kind $L$, that is, a nonnegative measure $L$ solving (almost everywhere) the convolution identity $K\ast L=1$.
	\end{hyp}
Note that, by Lebesgue's decomposition and dominated convergence theorems, the map $t\mapsto    L([0,t])$ is right-continuous in $\mathbb{R}_+$. Therefore, the convolution $K\ast L$ is right-continuous in $\mathbb{R}_+$, cfr. \cite[Corollary 6.2, Chapter 3]{gp}, and  the resolvent identity yields $L(\{0\})=K(0)^{-1}.$ By definition, $L$ is finite on compact sets, hence $K(0)>0$.\vspace{2mm}

	Define the open interval $I=(l,r)$, with $l,\,r\in \widebar{\mathbb{R}}$ such that $-\infty\le l<r\le\infty$. Take   a decreasing sequence  $(l_n)_n\subset \mathbb{R}$ and an increasing sequence $(r_n)_n\subset \mathbb{R}$ such that $l_n<r_n$, for every $n\in\mathbb{N}$, and 
	\[
		\lim_{n\to \infty}l_n=l,\qquad \lim_{n\to \infty}r_n=r.
	\]
	Letting $I_n=(l_n,r_n)$, it is clear that  $I=\cup_nI_n$.\vspace{2mm}\\
	 On a complete filtered probability space $(\Omega, \mathcal{F}, \mathbb{P};\mathbb{F}=(\mathcal{F}_t)_{t\ge0})$ where the filtration $\mathbb{F}$ satisfies the usual conditions, we consider an ($\mathbb{F}-$)adapted, $\widebar{I}=[l,r]-$valued continuous process $X=(X_t)_{t\ge 0}$ starting from $x_0\in I$. We define $S_n=\inf\{t\ge 0 : X_t\not\in I_n\}$ and $S=\inf\{t\ge 0 : X_t\not\in I\}$. Observe that $(S_n)_n$ is an increasing sequence of ($\mathbb{F}-$)stopping times converging to $S$. Given  a measurable drift $b\colon I\to \mathbb{R}$ such that 
		\begin{equation}\label{def_sol}
		\mathbb{P}\left(
		\int_{0}^{t\wedge S}|b(X_s)| \dd s<\infty
		\right)=1,	\quad t\ge0,
	\end{equation} 
	and
	a measurable locally bounded diffusion $\sigma\colon I\to\mathbb{R}$,	we suppose that $X$ satisfies the following stochastic Volterra equation (henceforth, SVE):
	\begin{equation}\label{SDE_VolterraC1}
		X_t=x_0+\int_{0}^{t}1_{\{s\le S_n\}}K(t-s)b(X_s)\,\dd s+ \int_{0}^{t}1_{\{s\le S_n\}}K(t-s)\sigma(X_s)\,\dd W_s,\quad  t\in [0,S_n],\, \mathbb{P}-\text{a.s.},
	\end{equation} 
for every $n\in \mathbb{N}$.
 Here $W=(W_t)_{t\ge 0}$ is a standard one--dimensional Brownian motion. We remark  that the integral processes appearing in the right-hand side of \eqref{SDE_VolterraC1} are well defined and continuous in $\mathbb{R}_+$. Indeed, the continuity of the integral in the Lebesgue measure follows from \eqref{def_sol}, \cite[Corollary 2.3, Chapter 2]{gp} and the continuity of $K$ in $\mathbb{R}_+$. As for the stochastic integral process, the continuity of its trajectories is given by  \cite[Lemma 2.4]{sergio}, which can be applied because $1_{\{\cdot\le S_n\}}X$ is bounded and $\sigma$ is locally bounded in $I$. Since the equality in \eqref{SDE_VolterraC1} holds in the random interval $[0,S_n]$, for every $n\in \mathbb{N}$, we say that 
 	\[
\text{ 	$X$ is a local solution of the SVE \,\, $X=x_0+K\ast b(X)+ K\ast (\sigma(X)\dd W)$\,\, in $[0,S)$},
 	\] 
 where $\ast$ denotes the convolution operation.\vspace{2mm}\\
 Our objective is to determine conditions, both  necessary and sufficient,  for $\mathbb{P}(S=\infty)=1$, namely, to ensure that $X$ does not exit from the interval $I$, see Subsection \ref{subsec_v}. Moreover, in Subsection \ref{suff_scalep}, we discuss sufficient conditions that can be employed to derive information on the $\sup$ and the $\inf$ of the solution process $X$ of \eqref{SDE_VolterraC1} in $[0,S)$. 
 \\
 Since $K(0)\in (0,\infty)$, i.e., the kernel $K$ does not explode at $0$, we can express the dynamics \eqref{SDE_VolterraC1} of $X$ in differential form, see Lemma \ref{le_semi}. This result is crucial for the subsequent arguments, as it shows that the stopped processes $X^{S_n},\,n\in\mathbb{N},$ are continuous semimartingales and enables us to use Itô's formula. 
 \begin{lemma}\label{le_semi}
 	For every $n\in \mathbb{N}$, the stopped process $X^{S_n}$ satisfies the following stochastic differential equation in $\mathbb{R}_+$:
 			\begin{multline}\label{SDE_lemma}
 				\dd X^{S_n}_t=1_{\{t\le S_n\}}\left[\left(K(0)b(X^{S_n}_t)+\frac{K'(0)}{K(0)}X^{S_n}_t\right)\dd t
 				+K(0)\sigma(X^{S_n}_t)\dd W_t\right]\\
 				+1_{\{t\le S_n\}}\left[\left(\dd\left(K'\ast L\right)\ast X^{S_n}\right)(t)-x_0\left(K'\ast L\right)(t)\right]\dd t
 				,
 			\end{multline}
 		 with initial condition $X^{S_n}_0=x_0.$
 \end{lemma}
\begin{proof}
Fix $n\in \mathbb{N}$. Consider the continuous semimartingale $Z^{(n)}$ starting at $0$ defined by
 $$Z^{(n)}_t=\int_{0}^{t}1_{\{s\le S_n\}}b(X^{S_n}_s)\dd s+ \int_{0}^{t}1_{\{s\le S_n\}}\sigma(X^{S_n}_s)\dd W_s,\quad t\ge 0.$$
Since $K\in C^1(\mathbb{R}_+)$, by the fundamental theorem of calculus $K=K(0)+K'\ast 1$. As a result, from \eqref{SDE_VolterraC1} we infer that, $\mathbb{P}-$a.s.,
\begin{multline*}
	X^{S_n}_t=x_0+\big(K\ast \dd Z^{(n)}\big)_t
	=
	x_0+K(0)Z^{(n)}_t+\big(\left(K'\ast 1\right)\ast \dd Z^{(n)}\big)_t
	\\=
	x_0+K(0)Z^{(n)}_t+\big( 1\ast \big(K'\ast\dd Z^{(n)}\big)\big)(t),\quad t\in[0,S_n],
\end{multline*}
where in the last step we also employ the associativity property of the stochastic convolution in \cite[Lemma 2.1]{sergio}. By Hypothesis \ref{Hyp1} (see also the subsequent comment) and \cite[Corollary 6.2, Chapter 3]{gp}, the map $t\mapsto (K'\ast L)(t)$ is right-continuous and of locally bounded variation. Therefore,  we can proceed as in the proof of \cite[Lemma 2.6]{sergio} to obtain from the previous equation, $\mathbb{P}-$a.s.,
\begin{multline*}
	X^{S_n}_t=x_0+K(0)Z^{(n)}_t+\frac{K'(0)}{K(0)}\left(1\ast X^{S_n}\right)(t) -x_0\left(1\ast \left(K'\ast L\right)\right)(t)\\+\left(1\ast \left(\dd(K'\ast L)\ast X^{S_n}\right)\right)(t),\quad t\in[0,S_n].
\end{multline*}
In differential form, this is precisely \eqref{SDE_lemma}, hence the proof is complete.
\end{proof}
We  define the modified drift  $\tilde{b}\colon I\to\mathbb{R}$ and the modified diffusion $\tilde{\sigma}\colon I\to\mathbb{R}$ by \[
	\tilde{b}(x)=K(0)b(x)+\frac{K'(0)}{K(0)}x,\qquad \tilde{\sigma}(x)=K(0)\sigma(x),\quad x\in I.
\]
In this way, for every $n\in\mathbb{N}$, Equation \eqref{SDE_lemma} reads
\begin{multline}\label{SDE_usare}
	\dd X^{S_n}_t=1_{\{t\le S_n\}}\left[\tilde{b}(X^{S_n}_t)\,\dd t+\tilde{\sigma}(X^{S_n}_t)\,\dd W_t\right]
	\\+1_{\{t\le S_n\}}\left[\left(\dd\left(K'\ast L\right)\ast \left(X^{S_n}-x_0\right)\right)(t)-x_0\frac{K'(0)}{K(0)}\right]\dd t
	,\quad
	X^{S_n}_0=x_0.
\end{multline}
Notice that the first two addends in the right-hand side of \eqref{SDE_usare} describe an It\^o diffusion, while the last drift term takes into account the path-dependent part of the dynamics, along with the deterministic  input curve $t\mapsto -x_0\frac{K'(0)}{K(0)}t$.
\begin{rem}\label{rem_CM}
	If $K\in C^1(\mathbb{R}_+)$ is a completely monotone kernel on $(0,\infty)$ not identically equal to $0$, then it satisfies Hypothesis \ref{Hyp1}. In particular, by \cite[Theorem 5.4, Chapter 5]{gp}, its resolvent of  the first kind $L$ admits the decomposition $L(\dd t)=K(0)^{-1}\delta_0(\dd t)+\rho(t)\,\dd t$, for some $\rho \in L^1_{\text{loc}}(\mathbb{R}_+)$ completely monotone on $(0,\infty)$.
	From the resolvent identity $K\ast L=1$ and \cite[Corollary 7.3, Chapter 3]{gp}, it then follows that $K'\ast L=-K(0)\rho$ in $(0,\infty)$. Taking the limit as $t\to 0+$, by \cite[Corollary 6.2, Chapter 3]{gp} we deduce that $K(0)\rho(0+)=-\frac{K'(0)}{K(0)}$. Thus, for every $n\in\mathbb{N},$ \eqref{SDE_lemma} can be rewritten as 
 		\begin{equation*}
 			\frac{\dd X^{S_n}_t}{K(0)}=1_{\{t\le S_n\}}
 			\left[\left(b(X^{S_n}_t)-\rho(0+)X^{S_n}_t\right)\dd t+\sigma(X^{S_n}_t)\,\dd W_t\right]+ 1_{\{t\le S_n\}}\left[\rho(0+)x_0-\left(\rho'\ast (X^{S_n}-x_0)\right)(t)\right]\dd t
 			,
 		\end{equation*}
 	with $X^{S_n}_0=x_0.$
\end{rem}	
By analogy with the classical case treated, for instance, in \cite[Section 5.5.C]{KS}, we impose the next assumptions of nondegeneracy and local integrability on the coefficients $\tilde{b}$ and $\tilde{\sigma}$ of \eqref{SDE_usare}.
\begin{hyp}\label{hyp_coeff}
	The map $\tilde{\sigma}\neq0$ almost everywhere in  $I$. Moreover, the functions ${\tilde{\sigma}^{-2}}$ and  $|\tilde{b}|{\tilde{\sigma}^{-2}}$ are locally integrable in $I$.
\end{hyp}
Under Hypothesis \ref{hyp_coeff}, for every $c\in I$ and $\beta,\,\gamma\in \mathbb{R}$, we  consider  the map $\tilde{b}_c(\cdot;\beta,\gamma)\colon I\to \mathbb{R}$ given by
\begin{equation}\label{def_bc}
	\tilde{b}_c(x;\beta,\gamma)=\tilde{b}(x)+\frac{K'(0)}{K(0)}\left(1_{\{x<c\}}\beta+1_{\{x\ge  c\}}\gamma\right),\quad x\in I.
\end{equation}
Such $\tilde{b}_c(\cdot;\beta,\gamma)$ can be interpreted as a shifted modified  drift, where the shift parameters are $\beta$ and $\gamma$. Clearly $\tilde{b}_c(\cdot;\beta,\gamma)$ satisfies Hypothesis \ref{hyp_coeff}, as well.
Notice that, when $x\in (l,c)$ [resp., $x\in [c,r)$], the value of $\tilde{b}_c(x;\beta,\gamma)$ depends only on $\beta$ [resp., $\gamma$]. In what follows, we will then write $\tilde{b}_c(x;\beta)$  [resp., $\tilde{b}_c(x;\gamma)$] for $x\in (l,c)$ [resp., $x\in [c,r)$] to simplify the notation.  We employ $\tilde{b}_c(\cdot;\beta,\gamma)$ and $\tilde{\sigma}$ to define the \emph{scale function} $p_c(\cdot;\beta,\gamma)\colon I\to \mathbb{R}$ by
 \begin{equation}\label{scale_func_def}
 	p_c(x;\beta,\gamma)=\int_{c}^{x}\exp\bigg\{-2\int_{c}^{y}\frac{\tilde{b}_c(z;\beta,\gamma)}{\tilde{\sigma}^2(z)}\dd z\bigg\}\dd y,\quad x\in I.
 \end{equation}
Also in this case, we only write $p_c(x;\beta)$ [resp. $p_c(x;\gamma)$] when $x\in (l,c)$ [resp.,  $x\in [c,r)$]. Observe that $p_c(\cdot; \beta, \gamma)$ is null at $x=c$ and (strictly) increasing in $I$. As a result, the limits $p_c(l+;\beta)$ and $p_c(r-;\gamma)$ always exist, finite or infinite. The mapping $p_c(\cdot;\beta,\gamma)$ belongs to $ C^1(I)$, and  a direct computation shows that 
\begin{equation}\label{eq_p}
	p_c''(x;\beta,\gamma)=-2\frac{\tilde{b}_c(x;\beta,\gamma)}{\tilde{\sigma}^2(x)}p_c'(x;\beta,\gamma),\quad \text{for a.e. }x\in I.
\end{equation}
As in the classical path-independent case, we use the scale function $p_c(\cdot;\beta,\gamma)$ to  define, by recursion, the sequence $(u_{c,n}(\cdot;\beta,\gamma))_{n\in\mathbb{N}_0}$ of $\mathbb{R}_+-$valued functions on $I$ where $u_{c,0}(\cdot;\beta,\gamma)\equiv1$ and 
\begin{equation}\label{v_def}
	u_{c,n}(x;\beta,\gamma)=2\int_{c}^{x}p'_c(y;\beta,\gamma)\left(\int_{c}^{y}\frac{u_{c,n-1}(z;\beta,\gamma)}{p'_c(z;\beta,\gamma)\tilde{\sigma}^2(z)}\dd z\right)\dd y
	,\quad x\in I,\,n\in\mathbb{N}.
\end{equation}
Of particular importance is the map
\begin{equation}\label{v_def2}
	v_c(x;\beta,\gamma)=u_{c,1}(x;\beta,\gamma),\quad x\in I.
\end{equation}
Indeed, for suitable values of $\beta$ and $\gamma$,   $v_c(\cdot;\beta,\gamma)$ will be employed to establish  necessary and sufficient conditions for $\mathbb{P}(S=\infty)=1$, see Theorems \ref{Feller_thm_nec}-\ref{Feller_thm} in Subsection \ref{subsec_v}. Notice that $v_c(\cdot;\beta,\gamma)$ is  decreasing in $(l,c)$, increasing in $(c,r)$ and equals to $0$ at $x=c$. As usual, we will write $v_c(x;\beta),\,u_{c,n}(x,\beta)$ [resp., $v_c(x;\gamma),\,u_{c,n}(x;\gamma)$] when $x\in (l,c)$ [resp., $x\in [c,r)$].\vspace{2mm}\\
We now present two technical results which analyze the dependence of the maps $p_c(\cdot;\beta,\gamma)$ and $v_c(\cdot;\beta,\gamma)$ on the parameter $c\in I$. Their proofs, which essentially rely on the definitions in \eqref{scale_func_def} and \eqref{v_def2}, are postponed to Appendix \ref{Appendix_tech}. 
\begin{lemma}\label{lemma_diff_c}
	Let $\beta,\,\gamma\in \mathbb{R}$ and consider $c_1,c_2\in I$ such that $c_1<c_2$.  Then, for every $x\in (l,c_1)$,
	\begin{equation}\label{p_dep_c}
		p_{c_2}(x;\beta)=p_{c_2}(c_1;\beta)+p'_{c_2}(c_1;\beta)p_{c_1}(x;\beta);\qquad 
		v_{c_2}(x;\beta)=v_{c_2}(c_1;\beta)	+
		p_{c_1}(x;\beta)
		v'_{c_2}(c_1;\beta)
		+
		v_{c_1}(x;\beta),
	\end{equation}
	and, for every $x\in (c_2,r)$,
	\begin{equation}\label{v_dep_c}
		p_{c_1}(x;\gamma)=p_{c_1}(c_2;\gamma)+p'_{c_1}(c_2;\gamma)p_{c_2}(x;\gamma);\qquad  v_{c_1}(x;\gamma)=v_{c_1}(c_2;\gamma)	+
		p_{c_2}(x;\gamma)
		v'_{c_1}(c_2;\gamma)
		+
		v_{c_2}(x;\gamma).
	\end{equation}
\end{lemma}

\begin{corollary}\label{fineteness_c}
	Let $\beta,\,\gamma\in \mathbb{R}$. Then the finiteness of $v_c(l+;\beta)$ or $v_c(r-;\gamma)$ does not depend on the choice of $c\in I$, namely, for every $c_1,\,c_2\in I$, 
	\[
\text{	$v_{c_1}(l+;\beta)=\infty$ [resp., $v_{c_1}(r-;\gamma)=\infty$] \,\,if and only if \,\,$v_{c_2}(l+;\beta)=\infty$ [resp., $v_{c_2}(r-;\gamma)=\infty$]}.
	\]
\end{corollary}

We conclude this preliminary part by formulating an additional requirement on the kernel $K$ and its resolvent of the first kind $L$. 
\begin{hyp}\label{hyp_increasing}
	The map $t\mapsto (K'\ast L)(t)$ is nonpositive and nondecreasing in $\mathbb{R}_+$
\end{hyp}
Under Hypothesis \ref{hyp_increasing}, the Lebesgue-Stieltjes measure $\dd (K'\ast L)(t)$ is nonnegative. This fact, along with the assumption $K'\ast L\le 0$, will be an essential tool for the arguments of Subsections \ref{suff_scalep}-\ref{subsec_v}, as it will allow to handle the path-dependent term in the dynamics \eqref{SDE_usare}. 
\begin{rem}
	Thanks to the computations in Remark \ref{rem_CM}, a kernel $K\in C^1(\mathbb{R}_+)$ which is completely monotone on $(0,\infty)$ satisfies Hypothesis \ref{hyp_increasing}.\\
	More generally, Hypothesis \ref{hyp_increasing} is met when $K'\le 0$  and $t\mapsto L([0,t])$ is concave in $\mathbb{R}_+$. Indeed, $K'\le 0$ together with the fact that $L$ is a nonnegative measure implies $K'\ast L\le 0$. The concavity of $L([0,\cdot])$ ensures that such a function is differentiable a.e. in $\mathbb{R}_+$, with a nonincreasing derivative. Differentiating the resolvent identity we obtain
	\[
	\left(K'\ast L\right)(t)=-K(0)\frac{\dd}{\dd t}L([0,t]),\quad \text{a.e. }t\ge 0,
	\]
	which shows that  the right-continuous map $K'\ast L$ (cfr. \cite[Corollary 6.2, Chapter 3]{gp}) is nondecreasing in $\mathbb{R}_+$. 
\end{rem}
\subsection{Sufficient conditions via \texorpdfstring{$p_c(\cdot;\beta,\gamma)$}{pc(·;β,γ)}}\label{suff_scalep}
In this subsection, we show that the values $p_c(l+;\beta)$ and $p_c(r-;\gamma)$ can be employed -- for a suitable choice of  $\beta,\gamma\in \mathbb{R}$ --  to describe the $\inf$ or the  $\sup$ of the solution process $X$ of \eqref{SDE_VolterraC1} in the random interval $[0,S)$. The results that we obtain are shown in Proposition \ref{suff_thm} and can be applied to  an half-bounded interval $I=(l,r)$, where $l>-\infty $ or $r<\infty$.
\begin{proposition}\label{suff_thm} Let $c\in I=(l,r)$ and $p_c(\cdot;\beta,\gamma),\,\beta,\gamma\in \mathbb{R},$ be as in \eqref{scale_func_def}. Assume Hypotheses \ref{Hyp1}-\ref{hyp_coeff}-\ref{hyp_increasing}.
\begin{enumerate}[label=(\roman*)]
	\item \label{p1_suff}Suppose that $-\infty\le l<r< \infty$. If $p_c(l+;-r)>-\infty$ and $p_c(r-;-r)=\infty$, then $$\mathbb{P}\left(\sup_{0\le t< S} X_t<r\right)=1;$$
	\item \label{dua_suff}
	Suppose that $-\infty< l<r\le \infty$. If $p_c(l+;-l)=-\infty$ and $p_c(r-;-l)<\infty$, then $$\mathbb{P}\left(\inf_{0\le t< S} X_t>l\right)=1.$$
\end{enumerate}
\end{proposition}
\begin{proof}
We only prove Point \ref{p1_suff}, as Point \ref{dua_suff} can be obtained by duality arguments. Suppose then that $-\infty\le l<r< \infty$, fix $c\in (l,r)$ and define, for every $\lambda,\,\rho\in I$ such that $\lambda<x_0<\rho$ and  $\lambda<c<\rho$, the stopping times 
\[
	\tau_\lambda=\inf\left\{t\ge0 : X_t\le \lambda\right\},\qquad 
		\tau_\rho=\inf\left\{t\ge0 : X_t\ge \rho\right\}.
\]
Letting $\tau_\lambda^\rho=\tau_\lambda\wedge \tau_\rho,$ from \eqref{SDE_usare}-\eqref{def_bc} we infer that 
	\begin{multline*}
		\dd X^{\tau_\lambda^\rho}_t=1_{\{t\le \tau_\lambda^\rho\}}\left[\tilde{b}_c(X_t;-r,-r)\,\dd t+\tilde{\sigma}(X^{}_t)\,\dd W_t\right]
		\\+1_{\{t\le \tau_\lambda^\rho\}}\left[\left(\dd\left(K'\ast L\right)\ast (X-r)^{}\right)(t)+(r-x_0)(K'\ast L)(t)\right]\dd t,
	\end{multline*}
with $X^{\tau_\lambda^\rho}_0=x_0.$
Then, applying a generalized It\^o's formula to the process $(p_c(X_{t\wedge \tau^\rho_\lambda};-r,-r))_{t\ge 0}$, see, e.g.,  \cite[Theorem 7.1 and Problem 7.3, Chapter 3]{KS}, by \eqref{eq_p},
\begin{multline*}
		\dd p_c\left(X^{\tau_\lambda^\rho}_t;-r,-r\right)=1_{\{t\le \tau_\lambda^\rho\}}p'_c\left(X_t;-r,-r\right)\Big[\tilde{\sigma}(X^{}_t)\,\dd W_t\\
		+
	\left(	\left(\dd\left(K'\ast L\right)\ast (X-r)^{}\right)(t)+(r-x_0)(K'\ast L)(t)\right)
	\dd t\Big], 
\end{multline*}
with $p_c\big(X^{\tau_\lambda^\rho}_0;-r,-r\big)=p_c(x_0;-r,-r)$. In the sequel, we are going to write $p_c(\cdot)$ instead of $p_c(\cdot;-r,-r)$ to shorten the notation. The right-hand side of the previous equation is the sum of a martingale (as $p'_c$ and $\tilde{\sigma}$ are locally bounded in $I$) and, thanks to Hypothesis \ref{hyp_increasing} 
 and recalling that $p_c'>0$ in $I$, of a nonpositive absolutely continuous process. As a result, taking expectations, we deduce that  
\[
	\mathbb{E}\left[p_c\left(X^{\tau_\lambda^\rho}_t\right)\right]\le p_c(x_0),\quad t\ge 0,
\]
whence, also using the fact that  $p_c$ is (strictly) increasing in $I$, 
 \[
		\mathbb{E}\left[p_c\left(X^{\tau_\lambda^\rho}_t\right)1_{\{\tau_\lambda^\rho<\infty\}}\right]\le p_c(x_0)-
	\mathbb{E}\left[p_c\left(X_t\right)1_{\{\tau_\lambda^\rho=\infty\}}\right]\le 
	p_c(x_0)-p_c(l+),\quad t\ge 0.
\]
Letting $t\to \infty$, the dominated convergence theorem and the continuity of $X$ in $\mathbb{R}_+$ yield 
\begin{align*}
p_c(\lambda)\mathbb{P}\left(\tau_\lambda^\rho=\tau_\lambda<\infty\right)
	+
	p_c(\rho)\mathbb{P}\left(\tau_\lambda^\rho=\tau_\rho<\infty\right)
	\le p_c(x_0)-p_c(l+).
\end{align*}
In particular, $\mathbb{P}\left(\tau_\lambda^\rho=\tau_\rho<\infty\right)\le p_c(\rho)^{-1}(p_c(x_0)-2p_c(l+))$. Thus, 
\[
\mathbb{P}\left(X_t=\rho,\,\text{for some }t\in[0,S)\right)= \lim_{\lambda\downarrow l}\mathbb{P}\left(\tau_\lambda^\rho=\tau_\rho<\infty\right)\le \frac{p_c(x_0)-2p_c(l+)}{p_c(\rho)},
\]
from which we conclude that $$\mathbb{P}\left(\sup_{0\le t< S} X_t=r\right)\le \lim_{\rho\uparrow r}\mathbb{P}\left(X_t=\rho,\,\text{for some }t\in[0,S)\right)= 0.$$
 This is equivalent to $\mathbb{P}\left(\sup_{0\le t< S} X_t<r\right)=1$,  which is the asserted equality.
\end{proof}

\subsection{Necessary and sufficient conditions via   \texorpdfstring{$v_c(\cdot;\beta,\gamma)$}{vc(·;β,γ)}}	\label{subsec_v}
In this subsection, more precisely in Theorems \ref{Feller_thm_nec}-\ref{Feller_thm}, we extend the classical Feller's test for explosions in \cite{WF} (see also \cite[Theorem 5.29, Chapter 5]{KS}) to the path-dependent case. To do this, we employ the function $v_c(\cdot;\beta,\gamma)$ defined in \eqref{v_def}-\eqref{v_def2}. 
\vspace{2mm}\\
We start by presenting the following preparatory lemma, which  compares the values of such $v_c(\cdot;\beta,\gamma)$ as $\beta$ and $\gamma$ vary in $\mathbb{R}$. Its proof is deferred to  Appendix \ref{Appendix_tech}.
\begin{lemma}\label{le_dip_betagamma}
	Fix $c\in I$ and assume that $K'(0)\le 0$. Consider $\beta_1,\,\beta_2\in \mathbb{R}$ and $\gamma_1,\,\gamma_2\in \mathbb{R}$ such that $\beta_1\ge \beta_2$ and $\gamma_1\le \gamma_2$. Then 
	$v_c(\cdot; \beta_1,\gamma_1)\le v_c(\cdot; \beta_2,\gamma_2)$ in $I$.\vspace{2mm}
	
	Furthermore, if $\int_{l}^{r}{\tilde{\sigma}^{-2}(z)}\,\dd z<\infty$, then $v_c(\cdot; \beta_2,\gamma_2)\le Cv_c(\cdot; \beta_1,\gamma_1)$ in $I$ for some constant $C>0$.
\end{lemma}

Consider $c\in I$ and $\beta,\,\gamma\in \mathbb{R}$. Recalling \eqref{v_def}, since $\tilde{b}_c(\cdot; \beta,\gamma)$ and $\tilde{\sigma}$ satisfy Hypothesis \ref{hyp_coeff}, by \cite[Lemma 5.26, Chapter 5]{KS}, the series 
\begin{equation}\label{definition_series}
	u_c(x;\beta,\gamma)=\sum_{n=0}^{\infty}u_{c,n}(x;\beta,\gamma)
\end{equation}  converges locally uniformly on $I$ to a differentiable function with locally absolutely continuous derivative. Such map $u_c(\cdot;\beta,\gamma)$ is decreasing in $(l,c)$, increasing in $(c,r)$ and solves the second order ODE
\begin{equation}\label{ODE_u}
\begin{cases}
u_c(x;\beta,\gamma)=\frac{1}{2}\tilde{\sigma}(x)u_c''(x;\beta,\gamma)+\tilde{b}_c(x;\beta,\gamma)u_c'(x;\beta,\gamma),\quad \text{for a.e. }x\in I,	\\
u_c(c;\beta,\gamma)=1,\,u_c'(c;\beta,\gamma)=0.
\end{cases}
\end{equation}
In addition, the following bounds hold (cfr. \cite[Equation (5.69), Chapter 5]{KS}):
\begin{equation}\label{control_u_KS}
	1+v_c(x;\beta,\gamma)\le u_c(x;\beta,\gamma)\le e^{v_c(x;\beta,\gamma)},\quad x\in I.
\end{equation}
As before, we will write $u_c(x;\beta)$ [resp., $u_c(x;\gamma)$] when $x\in (l,c)$ [resp., $x\in [c,r)$].
We remark that the above construction of the map $u_c(\cdot;\beta,\gamma)$ coincides with the one in the path-independent case, once we introduce, according to Hypothesis \ref{hyp_coeff}, the shifted modified drift $\tilde{b}_c(\cdot; \beta,\gamma)$, see \eqref{def_bc}, and the modified diffusion $\tilde{\sigma}.$ 
\vspace{2mm}

The next theorem gives a necessary condition for the solution process $X$ of \eqref{SDE_VolterraC1} not to exit from the interval $I$, i.e., for $\mathbb{P}(S=\infty)=1$. We emphasize that here, contrary to Proposition \ref{suff_thm}, we do not impose any restrictions on $I$, which can then be a generic real interval.
\begin{theorem}\label{Feller_thm_nec}
	Fix  $c\in I$ and assume Hypotheses \ref{Hyp1}-\ref{hyp_coeff}-\ref{hyp_increasing}. 
	Consider the stopping times $S_-=\inf\{t\ge 0: X_t=l\}$ and  $S_+=\inf\{t\ge 0: X_t=r\}$, so that $S=\min\{S_-,S_+\}$. 
\begin{enumerate}[label=(\roman*)]
	\item\label{nec_lower} If $\mathbb{P}(S_-=\infty)=1$, then  $v_c\big(l+;-{\beta}\big)=\infty$, for every $\beta>x_0$;
	\item \label{nec_upper} if $\mathbb{P}(S_+=\infty)=1$,  then $v_c\big(r-;-{\gamma}\big)=\infty$, for every $\gamma<x_0$.
\end{enumerate} 
Consequently,  
\[
	\text{$v_c\big(l+;-{\beta}\big)=\infty$, for every $\beta>x_0$, \quad and\quad  $v_c\big(r-;-{\gamma}\big)=\infty$, for every $\gamma<x_0$,}
\] is  a  necessary condition for $\mathbb{P}(S=\infty)=1$.
\end{theorem} 
Before the proof, notice that Theorem \ref{Feller_thm_nec} provides necessary conditions which depend on the starting point $x_0\in I$ of $X$. In particular, by virtue of Lemma \ref{le_dip_betagamma}, the requirement $v_c\big(l+;-{\beta}{}\big)=\infty$, for all $\beta>x_0,$ becomes more stringent as $x_0$ approaches $l$. On the other hand, the closer $x_0$ to $r$, the stricter the requirement $v_c\big(r-;-{\gamma}{}\big)=\infty,$ for all $\gamma<x_0$.
\begin{proof}
	We first prove Point \ref{nec_lower}. Given $x_0\in I$, we consider $c\in (x_0,r)$ and $\tau_c=\inf\{t\ge 0 : X_t\ge c\}$. Since  $K'\ast L\le0$ is nondecreasing in $\mathbb{R}_+$ by Hypothesis \ref{hyp_increasing}, for every continuous function $x\colon \mathbb{R}_+\to (-\infty,c]$ we have
	\begin{align}\label{nec_ser}
		\left(\dd\left(K'\ast L\right)\ast \left(x-x_0\right)\right)(h)
		&=
		\left(\dd\left(K'\ast L\right)\ast \left(x-c\right)\right)(h)+(c-x_0)\left((K'\ast L)(h)-(K'\ast L)(0)\right)
		\notag\\&\le
		(x_0-c)\frac{K'(0)}{K(0)},\quad h\ge 0.
	\end{align}
	From \eqref{SDE_usare}-\eqref{def_bc}, for every $n\in\mathbb{N}$, the dynamics in $\mathbb{R}_+$ of the $[l_n,c]-$valued stopped process $X^{S_n\wedge \tau_c}$ -- which is a continuous semimartingale --   are
	\begin{multline*}
		\dd X^{ S_n\wedge\tau_c }_t=1_{\{t\le  S_n\wedge\tau_c\}}\left[\tilde{b}_c\Big(X^{\tau_c}_t;-{c},-c{} \Big)\,\dd t+\tilde{\sigma}(X^{\tau_c}_t)\,\dd W_t\right]
		\\
		+1_{\{t\le  S_n\wedge\tau_c\}}\left[\left(\dd\left(K'\ast L\right)\ast \left(X^{\tau_c}-x_0\right)\right)(t)-\frac{K'(0)}{K(0)}\left(x_0-c\right)\right]\dd t.
	\end{multline*}
	Define ${Y}_t^{(n)}=e^{-(t\wedge S_n\wedge \tau_c)}u_c\big(X^{S_n\wedge \tau_c}_t;-c,-c\big),\,t\ge0.$ By the stochastic product rule,  the ODE \eqref{ODE_u} and Ito's lemma -- in a generalized form, since $u_c'\big(\cdot;-{c},-c{}\big)\in\,$AC$_\text{loc}((l,c])$, see for instance \cite[Theorem 7.1 and Problem 7.3, Chapter 3]{KS} -- we obtain
	\begin{align}\label{cancella1}
		\notag \!\!{Y}^{(n)}_t&=\int_{0}^{t}\!
		1_{\{h\le S_n\wedge \tau_c\}}e^{-h}u_c'\Big(X^{\tau_c}_h;-c,-c\Big)\left[\left(\dd\left(K'\ast L\right)\ast \left(X^{\tau_c}-x_0\right)\right)(h)-\frac{K'(0)}{K(0)}\left(x_0-c\right)\right]\!\dd h
		\\&\qquad +\int_{0}^{t}1_{\{h\le S_n\wedge \tau_c\}} e^{-h}u_c'\big(X^{\tau_c}_h;-{c},-c{}\big)\tilde{\sigma}(X^{\tau_c}_h)\,\dd W_h+u_c(x_0;-{c}{})\notag\\&
		\eqqcolon
		{\mathbf{\upperRomannumeral{1}}}^{(n)}_t+{\mathbf{\upperRomannumeral{2}}}^{(n)}_t+u_c(x_0;-{c}{}),
	\end{align}
	which holds for every $t\ge0,\,\mathbb{P}-\text{a.s.}$ Notice that  ${\mathbf{\upperRomannumeral{1}}}^{(n)}\ge0$ up to indistinguishability because of \eqref{nec_ser} and the fact that  $u'_c(\cdot;-{c}{})\le 0$ in $(l,c)$. Moreover, the stochastic integral process  ${\mathbf{\upperRomannumeral{2}}}^{(n)}$  is a  true martingale. Indeed, by the local boundedness of $\tilde{\sigma}$ and the continuity of $u'_c(\cdot;-c,-c)$ in $[l_n,c]$, the integrand process is bounded. Therefore,  taking expectations in \eqref{cancella1},
		\begin{equation}\label{proof_nec}
	u_c(x_0;-{c}{})\le \mathbb{E}\Big[{Y}^{(n)}_{t}\Big],\quad t\ge0,\,n\in\mathbb{N}.
\end{equation}
	 Suppose that $\mathbb{P}(S_-=\infty)=1$ and, arguing by contradiction, that $v_c\big(l+;-{c}{}\big)<\infty$. By \eqref{control_u_KS},
	\begin{equation*}
		u_c(x;-{c},-c{})\le C,\quad x\in (l,c],\text{ for some }C>0.
	\end{equation*}
In this setting, 	since $S\wedge \tau_c=S_+\wedge \tau_c=\tau_c$ almost surely, recalling the definition of $Y^{(n)}_t$, by the dominated convergence theorem   we compute the limit as  $n\to \infty$ in \eqref{proof_nec} to deduce that
	\begin{align*}
		&u_c(x_0;-{c}{})\le \mathbb{E}\Big[e^{-(t\wedge S\wedge \tau_c)}u_c(X_{t\wedge S\wedge \tau_c}+;-{c},-c )\Big]
		=
		\mathbb{E}\Big[
		e^{-(t\wedge \tau_c)}
		u_c(X_{t\wedge \tau_c};-{c},-c )
		\Big]
		\\&\quad
		=\mathbb{E}\Big[
		e^{-(t\wedge \tau_c)}
		u_c(X_{t\wedge \tau_c};-{c},-c{})\left(1_{\{\tau_c=\infty\}}+1_{\{\tau_c<\infty\}}\right)
		\Big]
		\le 
		Ce^{-t}+\mathbb{E}\Big[
		e^{-t\wedge \tau_c}
		u_c(X_{t\wedge \tau_c};-{c},-c )1_{\{\tau_c<\infty\}}
		\Big],
	\end{align*}
	which holds for every $t\ge0$.
	Letting $t\to\infty$, once again by dominated convergence we obtain the inequality 
	\[
u_c(x_0;-{c})\le 
	\mathbb{E}\Big[
	e^{-\tau_c}
	u_c(X_{ \tau_c};-{c})1_{\{\tau_c<\infty\}}
	\Big]
	\le u_c(c;-{c}),
	\] 
	which is absurd because $u_c(\cdot;-{c},-c)$ is (strictly) decreasing in $(l,c]$. Thus, $$v_c(l+;-{c})=\infty.$$ 
	In turn, this equality combined with Lemma \ref{le_dip_betagamma} yields $v_c(l+;-{\beta})=\infty$, for every $\beta\ge c$. Since we chose a generic $c\in (x_0,r)$, by Corollary \ref{fineteness_c} we conclude that 
	\[
	v_c(l+;-{\beta}{})=\infty,\quad
	\beta>x_0,\,c\in I,
	\]
	proving Point \ref{nec_lower}.\vspace{2mm}
	\\
	As for Point \ref{nec_upper}, the necessity of the condition $$v_c(r-;-{\gamma}{})=\infty,\quad \gamma<x_0,\,c\in I,$$ 
	for $\mathbb{P}(S_+=\infty)=1$ can be inferred similarly, this time arguing with $ c\in (l,x_0)$, defining $\tau_c=\inf\{t\ge 0 : X_t\le c\}$ and using the fact that $u_c(\cdot;-c)$ is (strictly) increasing in $[c,r)$. The proof is then complete.
\end{proof}
In the following theorem, we use the maps $v_c(\cdot;\beta,\gamma)$, for certain $\beta,\gamma\in \mathbb{R}$, to establish sufficient conditions  guaranteeing that $\mathbb{P}(S=\infty)=1$.  The hypotheses that we formulate in Theorem \ref{Feller_thm} are stronger than the necessary conditions in Theorem \ref{Feller_thm_nec} and differ from the classical ones in \cite{WF}. We refer to Remarks \ref{rem_cl_suff}-\ref{classical_case} and Corollary \ref{coro_nec_suff} for more details. 
\begin{theorem}\label{Feller_thm}
	 Fix $c\in I$ and assume Hypotheses \ref{Hyp1}-\ref{hyp_coeff}-\ref{hyp_increasing}. Consider the stopping times $S_-=\inf\{t\ge 0: X_t=l\}$ and  $S_+=\inf\{t\ge 0: X_t=r\}$, so that $S=\min\{S_-,S_+\}$. 
	\begin{enumerate}[label=(\roman*)]
	 	\item\label{suff_lower} If $\lim_{n\to \infty}v_c(l_n;-l_n)=\infty$, then $\mathbb{P}(\{S<S_-\}\cup\{S_-=\infty\} )=1$;
	 	\item\label{suff_upper}  if $\lim_{n\to \infty}v_c(r_n;-r_n)=\infty$, then $\mathbb{P}(\{S<S_+\}\cup\{S_+=\infty\} )=1$.
	 \end{enumerate}
	 
	 Consequently, the condition $$\lim_{n\to \infty}v_c(l_n;-l_n)=\lim_{n\to \infty}v_c(r_n;-r_n)=\infty$$ is sufficient for  $\mathbb{P}(S=\infty)=1$.
\end{theorem}
\begin{proof}
	Consider $c\in I$, where $I=(l,r)$ for some $-\infty\le l<r\le \infty$. Recalling \eqref{SDE_usare} and \eqref{def_bc}, 
	the dynamics of the stopped process $X^{S_n}$ in $\mathbb{R}_+$ read, for every $n\in \mathbb{N}$,
	\begin{multline}\label{negative}
		\dd X^{S_n}_t=1_{\{t\le S_n\}}\left[\tilde{b}_c(X_t;-l_n, -r_n)\,\dd t+\tilde{\sigma}(X_t)\,\dd W_t\right]
		\\
		+1_{\{t\le S_n\}}\left[\left(\dd\left(K'\ast L\right)\ast \left(X-x_0\right)\right)(t)+\frac{K'(0)}{K(0)}\left(l_n1_{\{X_t<c\}}+r_n1_{\{X_t\ge c\}}
		-x_0\right)\right]\dd t,
	\end{multline}
	with  $X^{S_n}_0=x_0$. 
	Define the continuous semimartingale  $\widebar{Z}^{(n)}_t=u_c(X^{S_n}_t;-l_n, -r_n),\,t\ge0$. 
	 An application of the generalized Ito's lemma in  \cite[Theorem 7.1 and Problem 7.3, Chapter 3]{KS}, along with the ODE \eqref{ODE_u} yields
	\begin{multline}\label{eq_pr_Fe}
		\dd \widebar{Z}^{(n)}_t=1_{\{t\le S_n\}}\left(u_c\big(X_t;-l_n, -r_n\big)\dd t +u'_c\big(X_t;-l_n, -r_n\big)\tilde{\sigma}(X_t)\,\dd W_t\right)
		\\+
		1_{\{t\le S_n\}}u_c'\big(X_t;-l_n, -r_n\big)\left[\left(\dd\left(K'\ast L\right)\ast \left(X-x_0\right)\right)(t)+\frac{K'(0)}{K(0)}\left(l_n 1_{\{X_t<c\}}+r_n1_{\{X_t\ge c\}}-x_0\right)\right]\dd t,
	\end{multline}
	with initial condition  $\widebar{Z}^{(n)}_0=u(x_0;-l_n, -r_n)$. Then, if we introduce the process $\widebar{Y}_t^{(n)}=e^{-t\wedge S_n}\widebar{Z}_t^{(n)},\,t\ge0$, the stochastic product rule implies that
	\begin{align}\label{supermg_Y}
		\notag \widebar{Y}^{(n)}_t\!&=\!\int_{0}^{t}\!\!
		1_{\{h\le S_n\}}e^{-h}u_c'(X_h;-l_n, -r_n)\!\left[\left(\dd\left(K'\ast L\right)\ast \left(X-x_0\right)\right)\!(h)\!+\frac{K'(0)}{K(0)}\left(l_n1_{\{X_h<c\}}\!+r_n1_{\{X_h\ge  c\}}\!-x_0\right)\right] \!\dd h
	\\&	\notag\quad +\int_{0}^{t}1_{\{h\le S_n\}} e^{-h}u_c'(X_h;-l_n, -r_n)\tilde{\sigma}(X_h)\,\dd W_h+u_c(x_0;-l_n, -r_n)\\&\eqqcolon
		\widebar{\mathbf{\upperRomannumeral{1}}}^{(n)}_t+\widebar{\mathbf{\upperRomannumeral{2}}}^{(n)}_t+u_c(x_0;-l_n, -r_n),\quad t\ge0,\,\mathbb{P}-\text{a.s.}
	\end{align}
	We concentrate on $\widebar{\mathbf{\upperRomannumeral{1}}}^{(n)}_{t}\eqqcolon\int_{0}^{t}g^{(n)}(h)\,\dd h,\,t\ge0$, with the aim of demonstrating that it is a nonpositive supermartingale when $n\in \mathbb{N}$ is large enough. To do this, it suffices to show that the integrand $g^{(n)}\le 0$ for $n$ large enough. We then consider $\bar{n}\in \mathbb{N}$ sufficiently big such that $x_0\in I_n=(l_n,r_n)$, for every $n\ge \bar{n}$. By Hypothesis~\ref{hyp_increasing}, $K'\ast L\le0$ is nondecreasing in $\mathbb{R}_+$, whence, for every $n\ge \bar{n}$ and every continuous map $x\colon \mathbb{R}_+\to I_n$, 
	\begin{align*}
		\left(\dd\left(K'\ast L\right)\ast \left(x-x_0\right)\right)(h)-x_0\frac{K'(0)}{K(0)}
		&=(r_n-x_0)(K'\ast L)(h)+
		(\dd\left(K'\ast L\right)\ast \left(x-r_n\right))(h)-r_n\frac{K'(0)}{K(0)}
		\\&\le
		-r_n\frac{K'(0)}{K(0)},\quad h\ge 0.
	\end{align*}
	Analogously,
	\begin{align*}
		\left(\dd\left(K'\ast L\right)\ast \left(x-x_0\right)\right)(h)-x_0\frac{K'(0)}{K(0)}
	&=(l_n-x_0)(K'\ast L)(h)+
	(\dd\left(K'\ast L\right)\ast \left(x-l_n\right))(h)-l_n\frac{K'(0)}{K(0)}
	\\&\ge
	-	l_n\frac{K'(0)}{K(0)},\quad h\ge 0.
	\end{align*}
	Since  $u'_c(\cdot;-l_n)\le0$ in $(l,c)$ and $u'_c(\cdot;-r_n)\ge0$ in $[c,r)$, the two previous equations show that $g^{(n)}\le 0$, as desired.
 \\Going back to \eqref{supermg_Y} and recalling that $u_c(\cdot;-l_n,-r_n)$ is (strictly) positive on $I$, which entails that $\widebar{Y}^{(n)}_t\ge0$ by definition, we observe that the process  $u_c(x_0;-l_n,-r_n)+\widebar{\mathbf{\upperRomannumeral{2}}}^{(n)}_t,\,t\ge0,$ is a nonnegative local martingale for $n\ge \bar{n}$. Indeed,  \eqref{supermg_Y} gives
	$$u_c(x_0;-l_n,-r_n)+\widebar{\mathbf{\upperRomannumeral{2}}}^{(n)}_t=\widebar{Y}_t^{(n)}-\widebar{\mathbf{\upperRomannumeral{1}}}^{(n)}_t,\quad t\ge0,\,\mathbb{P}-\text{a.s.}$$
	Hence it is a supermartingale, and so is $\widebar{Y}^{(n)}$. By Fatou's lemma, we can claim that also the $[0,\infty]-$valued $\liminf-$process $\widebar{Y}=(\widebar{Y}_t)_{t\ge0}$ given by
	\begin{equation}\label{widebary}
	\widebar{Y}_t=
	\liminf_{n\to\infty}\widebar{Y}^{(n)}_t=
	e^{-(t\wedge S)}\liminf_{n\to \infty}u_c\big(X_t^{S_n};-l_n,-r_n\big)
	,\quad t\ge0,
	\end{equation}
	is a supermartingale. Note that, for every $t\ge0$, the random variable $\widebar{Y}_t$ is finite $\mathbb{P}-$a.s.,  because $t\mapsto \mathbb{E}[\widebar{Y}_t]$ is nonincreasing and starts at a finite deterministic value $\widebar{Y}_0$. Indeed, by \eqref{control_u_KS} and \eqref{widebary}, 
	\[
		\widebar{Y}_0= \liminf_{n\to \infty}u_c(x_0;-l_n,-r_n)\le \liminf_{n\to \infty} e^{v_c(x_0;-l_n,-r_n)}\le e^{v_c(x_0;-l_1,-r_1)}<\infty,
	\]  
	where we apply Lemma \ref{le_dip_betagamma} for the second-to-last inequality. 
	\vspace{2mm}\\
	We now prove the implication in Point \ref{suff_lower}: the one in Point \ref{suff_upper} can be shown with analogous arguments. By contradiction, suppose that there exists $t>0$ such that $\mathbb{P}(\{S=S_-\}\cap \{S_-<  t\})>0$. Then, writing $\Omega'=\{S=S_-\}\cap \{S_-<  t\}$,  
	\[
		X^{S_-}_t(\omega)=l,\quad \omega\in \Omega'.
	\]
	 Observe that, by the continuity of the process $X$, for every $\omega\in \Omega'$ there is an $N(\omega)\in \mathbb{N}$ so big  that $l_n<c$ and that 
	\[
		u_c\big(X^{S_n}_t(\omega);-l_n,-r_n\big)=u_c(l_n;-l_n)\ge 1+v_c(l_n;-l_n), \quad n\ge N(\omega),
	\]
	where we use   \eqref{control_u_KS} for the inequality.
	Thus, from \eqref{widebary} and the hypotheses of this theorem,
	\[
		\widebar{Y}_t(\omega)= e^{-S(\omega)}\liminf_{n\to \infty}u_c\big(X_t^{S_n}(\omega);-l_n,-r_n\big)\ge
		e^{-S(\omega)} \lim_{n\to \infty}v_c(l_n;-l_n)=\infty,\quad \omega\in \Omega',
	\]
	which is absurd because we have argued that $\widebar Y_t<\infty$ almost surely.
	Therefore,
	$$\mathbb{P}(\{S=S_-\}\cap \{S_-<  t\})=0 \quad \text{for every $t\ge 0$}.$$
	 Consequently,
		\[
		\mathbb{P}(\{S=S_-\}\cap \{S_-<\infty\})\le 
		\mathbb{P}\Big(\{S=S_-\}\cap \Big(\bigcup_{n\in\mathbb{N}}\{S_-< n\}\Big)\Big)=0,
	\]
	 which completes the proof.
\end{proof}
\begin{rem}\label{rem_cl_suff}
	Theorem \ref{Feller_thm} can be applied to general, possibly unbounded intervals. Roughly speaking, in the proof, the fact that the endpoints of $I$ might be infinite is handled  by adding terms depending on the (finite) quantities $l_n$ and $r_n$. This enables to control  the sign of the  path-dependent term in the dynamics \eqref{SDE_usare} of the stopped processes $X^{S_n}$.  Such an argument requires a sufficient condition involving the sequences $(v_c(l_n;-l_n))_n$ and $(v_c(r_n;-r_n))_n$, contrary to the classical formulation of Feller's test for explosions, see \cite{WF} and \cite{KS}.  \vspace{2mm}\\However, when $l>-\infty$ or $r<\infty$, it is possible to impose stronger requirements which satisfy the hypotheses of Theorem \ref{Feller_thm} and resemble more the classical ones.
		More specifically, suppose that $l>-\infty$ and that $v_c(l+;-l)=\infty$. Then, by Lemma \ref{le_dip_betagamma},
	\[
		\liminf_{n\to \infty}v_c(l_n;-l_n)\ge \liminf_{n\to \infty}v_c(l_n;-l)=v_c(l+;-l)=\infty,
	\]
	whence $\lim_{n\to \infty}v_c(l_n;-l_n)=\infty$. Analogously, if $r<\infty$, the condition $\lim_{n\to \infty }v_c(r_n;-r_n)=\infty$ is  implied by $v_c(r-;-r)=\infty$.
\end{rem}
{\color{black}Under stronger assumptions on the modified diffusion $\tilde{\sigma}$, it is possible to combine Theorems \ref{Feller_thm_nec} and \ref{Feller_thm} to deduce a necessary and sufficient condition for $\mathbb{P}(S=\infty)=1$. More precisely,  by Lemma \ref{le_dip_betagamma} 
	and Remark \ref{rem_cl_suff}, the next result holds.
\begin{corollary}\label{coro_nec_suff}
	 Assume Hypotheses \ref{Hyp1}-\ref{hyp_coeff}-\ref{hyp_increasing}. 
	Consider the stopping times $S_-=\inf\{t\ge 0: X_t=l\}$ and  $S_+=\inf\{t\ge 0: X_t=r\}$, so that $S=\min\{S_-,S_+\}$. Furthermore, suppose that 
		 $I\subset \mathbb{R}$ is a bounded interval and that $\tilde{\sigma}^{-2}$ is integrable in $I$.
Then, given $c\in I$ and $\beta,\,\gamma\in \mathbb{R}$,
\begin{enumerate}[label=(\roman*)]
	\item \label{suffnec_lo} the condition $v_c(l+;\beta)=\infty$ is necessary for $\mathbb{P}(S_-=\infty)=1$ and sufficient for $\mathbb{P}(\{S<S_-\}\cup\{S_-=\infty\} )=1$;
	\item \label{suffnec_up} the condition $v_c(r-;\gamma)=\infty$ is necessary for $\mathbb{P}(S_+=\infty)=1$ and sufficient for $\mathbb{P}(\{S<S_+\}\cup\{S_+=\infty\} )=1$.
\end{enumerate}
In particular, $$\text{$\mathbb{P}(S=\infty)=1$ \quad if and only if \quad $v_c(l+;\beta)=v_c(r-;\gamma)=\infty.$}$$
\end{corollary}
\begin{rem}\label{classical_case}
	In the classical case $K\equiv K(0)$, for some positive constant $K(0)>0$, \eqref{SDE_VolterraC1} reduces to a standard It\^o diffusion without path-dependent component. In this context, since $K'(0)=0$, it directly follows from the definitions in \eqref{scale_func_def} and \eqref{v_def}-\eqref{v_def2} that  $p_c(\cdot; \beta,\gamma)=p_c(\cdot;0,0)$ and $v_c(\cdot;\beta,\gamma)=v_c(\cdot;0,0)$, which means that the maps $p_c(\cdot; \beta,\gamma)$ and $v_c(\cdot;\beta,\gamma)$ do not depend on $\beta$ or $\gamma$. As a result, the conclusions of Corollary~\ref{coro_nec_suff} continue to hold even removing the hypotheses of boundedness for the interval $I$ and of integrability  for the modified diffusion $\tilde{\sigma}$.  Hence, in this framework, Corollary \ref{coro_nec_suff} coincides with the classical Feller's test for explosions in \cite{WF}.
	
	 We also notice that positive constants are the only kernels for which this argument applies. Indeed, if we consider a map $K\in C^1(\mathbb{R}_+)$, with $K'(0)=0$, satisfying Hypotheses \ref{Hyp1}-\ref{hyp_increasing}, then $(K'\ast L)(0)=0$. Since $K'\ast L\le 0$ is nondecreasing by Hypothesis \ref{hyp_increasing}, it follows that  $K'\ast L=0$ in $\mathbb{R}_+$. Therefore, convolving with $K$, the resolvent identity and the fundamental theorem of calculus imply that $K(t)-K(0)=(K'\ast 1)(t)=0$ for every  $t\ge0$,  that is, $K\equiv K(0)$.
\end{rem}

}

\section{Examples}\label{sec_ex}
The purpose of  this section is to apply the results of the previous analysis, especially the necessary conditions in Theorem \ref{Feller_thm_nec} and the sufficient ones in Theorem \ref{Feller_thm}, to three specifications of the dynamics \eqref{SDE_VolterraC1} that are relevant in the literature.  In particular, in Subsection \ref{subsection_CIR} our  focus lies on Volterra CIR processes, for which the state space ($\mathbb{R}_+$) boundary attainment  coincides with the strict positivity of the trajectories. Subsection \ref{subsection_JV} deals instead  with Volterra Jacobi processes, characterized by having a bounded interval as their state space. Finally, in Subsection \ref{subsection_power}, we consider  Volterra power-type diffusion processes, for which the interest is in the possibility of having blow-ups to $\pm\infty$ in finite time.\\
In Subsections \ref{subsection_CIR} and \ref{subsection_JV}, in order to guarantee the existence of  weak solutions to the equations under scrutiny (see \eqref{VolterraCIR} and \eqref{VJac} below),  we suppose that the nonexplosive kernel $K$ satisfies the following requirement in addition to Hypotheses \ref{Hyp1}-\ref{hyp_increasing}.
\begin{hyp}\label{hyp_res}
	The kernel $K$ is nonnegative and nonincreasing. Moreover,  its resolvent of the first kind $L$ is nonincreasing, i.e., the map $ s\mapsto L([s,s+t])$ is nonincreasing for every $t\ge0$.
\end{hyp}
Hypothesis \ref{hyp_res} -- a standard assumption in the literature on Volterra processes in both continuous (see \cite{Pol_sergio, sergio}) and jump (see \cite{ABI_j, BPL}) cases --  is satisfied, for example, by completely monotone kernels.

\subsection{The Volterra CIR process}\label{subsection_CIR}
Let $\kappa, \theta,\sigma>0$ and consider the SVE
\begin{equation}\label{VolterraCIR}
			X_t=x_0+\int_{0}^{t}K(t-s)\kappa(\theta-X_s)\,\dd s+ \sigma \int_{0}^{t}K(t-s)\sqrt{X_s}\,\dd W_s,\quad  t\ge 0, \,\mathbb{P}-\text{a.s.},
\end{equation}
where the starting point $x_0>0$.
Since $K$ belongs to $ C^1(\mathbb{R}_+)$ and admits a resolvent of the first kind satisfying Hypothesis \ref{hyp_res}, the weak existence of a continuous nonnegative solution of \eqref{VolterraCIR} is guaranteed by \cite[Theorem 3.6]{sergio}.  In fact, \eqref{VolterraCIR} admits a pathwise unique strong solution, as shown in \cite[Proposition B.3]{ABE1} by adapting the original strategy in \cite[Theorem 1]{YW}. We call such solution $X=(X_t)_{t\in\mathbb{R}_+}$ a \emph{Volterra CIR process}. Using the results of Section \ref{main}, we aim to determine conditions on the parameters $\kappa,\,\theta,\,\sigma$ and on the initial state $x_0$ under which the trajectories of $X$ attain (or do not attain) the value $0$. 
\vspace{2mm}

We fix  $l=0$, $r=\infty$ and define the maps $\tilde{b},\,\tilde{\sigma}\colon (0,\infty)\to \mathbb{R}$ by
\[
	\tilde{b}(x)=K(0)\kappa(\theta-x)+\frac{K'(0)}{K(0)}x,\qquad \tilde{\sigma}(x)=K(0)\sigma\sqrt{x},\quad x\in (0,\infty).
\]
Since $X$ is a global solution of \eqref{VolterraCIR} and $r=\infty$, we notice that $$S=S_-=\inf\{t\ge 0:X_t=0\}.$$
Let $c=1$ and, to simplify the notation, $C=2(K(0)\sigma)^{-2}$. Given $\beta,\,\gamma\in\mathbb{R}$, the derivative of the scale function $p_1(\cdot;\beta,\gamma)\colon(0,\infty)\to \mathbb{R}$ associated with the dynamics \eqref{VolterraCIR} is (cfr. the definitions in \eqref{def_bc}-\eqref{scale_func_def}) 
\begin{align}\label{p1_CIR-}
\notag	p'_1(x;\beta)&=\exp\left\{ -\frac{2}{K(0)^2\sigma^2}\int_{1}^{x}\frac{1}{y}\left(K(0)\kappa(\theta-y)+\frac{K'(0)}{K(0)}(y+\beta)\right) \dd y\right\}
	\\&=
\notag	\exp\left\{ -C\left(\left(\frac{K'(0)}{K(0)}-K(0)\kappa\right)(x-1)+\left(K(0)\kappa\theta+\beta\frac{K'(0)}{K(0)}\right)\log x\right)\right\}
	\\&
	=
	x^{-C\left(K(0)\kappa\theta+\beta{K'(0)}{K^{-1}(0)}\right)}
	\exp\left\{ -C\left(\frac{K'(0)}{K(0)}-K(0)\kappa\right)(x-1)\right\}
	,\quad x\in (0,1).
\end{align}
By analogous computations,  
\begin{align}\label{p1_CIR+}
	p'_1(x;\gamma)=x^{-C\left(K(0)\kappa\theta+\gamma{K'(0)}{K^{-1}(0)}\right)}
	\exp\left\{ -C\left(\frac{K'(0)}{K(0)}-K(0)\kappa\right)(x-1)\right\}
	,\quad x\in [1,\infty).
\end{align}
As for the nonnegative function $v_1(\cdot;\beta,\gamma) \colon (0,\infty)\to\mathbb{R}_+$ defined in \eqref{v_def2}, we suppose that $K(0)\kappa\theta+\beta{K'(0)}{K^{-1}(0)}\neq 0$. Then  the following lower bound for $v_1(\cdot;\beta)$ holds in the interval $(0,1)$:
\begin{align}\label{lower_b_V}
	v_1(x;\beta)&\notag=C\int_{x}^{1}p'_1(y;\beta)\left(\int_{y}^{1}z^{
C\left(K(0)\kappa\theta+\beta{K'(0)}{K^{-1}(0)}\right)-1	
}
\exp\left\{ {C(z-1)}\left(\frac{K'(0)}{K(0)}-K(0)\kappa\right)\right\}
\dd z \right)\dd y\\&\notag
\ge 
C_1
\left(K(0)\kappa\theta+\beta{K'(0)}{K^{-1}(0)}\right)^{-1	}
\int_{x}^{1}p'_1(y;\beta)
\left(1-y^{C\left(K(0)\kappa\theta+\beta{K'(0)}{K^{-1}(0)}\right)}\right)
\dd y
\\
&
\ge C_1\int_{x}^{1}\left|y^{-C\left(K(0)\kappa\theta+\beta{K'(0)}{K^{-1}(0)}\right)}-1\right|\dd y
,\quad x\in (0,1).
\end{align}
 We can also establish an analogous upper bound on the same interval, given by
\begin{equation}\label{upper_b_V}
	 v_1(x;\beta)\le 
  C_2\int_{x}^{1}\left|y^{-C\left(K(0)\kappa\theta+\beta{K'(0)}{K^{-1}(0)}\right)}-1\right|\dd y
	,\quad x\in (0,1).
\end{equation}
Here $C_1,\,C_2>0$ are  positive constants, possibly depending on $K,\,\kappa,\,\theta,\,\sigma$ and $\beta$, which are  allowed to change from line to line. Combining these estimates with the necessary and sufficient conditions in Subsection \ref{subsec_v}, we obtain the next theorem. 
\begin{theorem}\label{Fel_CIR}
	Assume Hypotheses \ref{Hyp1}-\ref{hyp_increasing}-\ref{hyp_res}.
	\begin{enumerate} [label=(\roman*)]
		\item \label{nec_CIR} Suppose that $K'(0)<0$. Then, the condition $x_0\ge {K(0)^2}(2|K'(0)|)^{-1}(K(0)\sigma^2-2\kappa\theta)$ is necessary for $\mathbb{P}(S=\infty)=1$.
		\item \label{suff_CIR} The condition $2\kappa\theta\ge K(0)\sigma^2$ is sufficient for  $\mathbb{P}(S=\infty)=1$.
	\end{enumerate}
\end{theorem}
\begin{proof}
	We start by proving Point \ref{nec_CIR} by contradiction. Thus, we consider $$x_0<{K(0)^2}(2|K'(0)|)^{-1}(K(0)\sigma^2-2\kappa\theta)
	\quad \Longleftrightarrow \quad 
C\left[ K(0)\kappa \theta+x_0 \frac{|K'(0)|}{K(0)}\right]< 1
	.$$
	We can then select a  $\bar{\beta}>x_0>0$ such that $ K(0)\kappa \theta+\bar{\beta} {K(0)}|K'(0)|^{-1}\neq 0$ and $\bar{\beta}-x_0>0$ is  so small that the previous bounds hold replacing $\bar{\beta}$ for $x_0$. 
	As a result, by \eqref{upper_b_V}, for  a positive constant $C_{\bar{\beta}}$,
	\[
	v_1(0+;-\bar{\beta})\le C_{\bar{\beta}}\int_{0}^{1}\left|y^{-C\left(K(0)\kappa\theta+\bar{\beta}{|K'(0)|}{K^{-1}(0)}\right)}-1\right|\dd y<\infty,
	\]
where we use the fact that $|K'(0)|=-K'(0)$. Hence the necessary condition in    Theorem \ref{Feller_thm_nec} \ref{nec_lower} is violated, which yields $\mathbb{P}(S<\infty)>0$. This is absurd, demonstrating Point \ref{nec_CIR}.\vspace{2mm}

As for Point \ref{suff_CIR}, we assume $2\kappa\theta \ge K(0)\sigma^2$ and observe that, by   \eqref{lower_b_V}, 
\[
	v_1(0+;0)\ge \widetilde{C}\int_{0}^{1}\left(y^{-CK(0)\kappa\theta}-1\right)\dd y=\infty,\quad \text{for some }\widetilde{C}>0.
\]
Therefore, thanks to Remark \ref{rem_cl_suff},  we invoke Theorem \ref{Feller_thm} \ref{suff_lower}  to conclude that  $\mathbb{P}(S=\infty)=1$. The proof is then complete. 
\end{proof}
\begin{rem}\label{classical_CIR}
	When $K\equiv K(0),$ for some positive constant $K(0)>0$, we can strengthen the results in Theorem \ref{Fel_CIR} to obtain a necessary and sufficient condition for $\mathbb{P}(S=\infty)=1$. More specifically, 
	\begin{equation}\label{fel_con_CIR}
		\text{$\mathbb{P}(X_t>0,\,t\in\mathbb{R}_+)=1\quad $ if and only if $\quad 2\kappa \theta\ge K(0)\sigma^2$.}
	\end{equation}
Indeed, the fact that  $2\kappa \theta\ge K(0)\sigma^2 \,\Longrightarrow\, \mathbb{P}(X_t>0,\,t\in\mathbb{R}_+)=1$ is guaranteed by Theorem \ref{Fel_CIR} \ref{suff_CIR}.  The reverse implication is instead obtained by Corollary \ref{coro_nec_suff} \ref{suffnec_lo} and Remark \ref{classical_case}. Notice that, choosing $K\equiv1$, \eqref{fel_con_CIR} reduces to the well--known Feller condition for CIR processes. 

	In addition, if $2\kappa\theta<K(0)\sigma^2$, then not only $X$ attains the value $0$ with positive probability, but also $\mathbb{P}(\sup_{t\in [0,S)}X_t<\infty)=1$, meaning that its trajectories are bounded almost surely in $[0,S)$. This follows from Proposition \ref{suff_thm} \ref{p1_suff} -- which can be applied as  $K'(0)=0$, hence $p_c(\cdot;\beta,\gamma)$ does not depend on $\beta$ or $\gamma$ -- and \eqref{p1_CIR-}-\eqref{p1_CIR+}.
\end{rem}
\paragraph{Approximating the fractional kernel} In the  mathematical finance literature on stochastic volatility models, particular attention has been recently devoted to the dynamics \eqref{VolterraCIR} where $K$ is the fractional kernel, namely $K(t)=K_\alpha(t)=\Gamma(\alpha)^{-1}t^{\alpha-1}$, $t>0$, with $\alpha\in (0,1)$. In fact, considering $K=K_\alpha$, \eqref{VolterraCIR} describes the spot variance in the so-called rough Heston model, see \cite{rh1,rh2}. Note that, when $\alpha\in (0,\frac{1}{2}]$, $K_\alpha$ is only locally integrable (and not locally square-integrable), hence \eqref{VolterraCIR} has to be understood in an integral form, see \cite{ABI_j}. As already mentioned in  Introduction \ref{sec_Intro}, our approach does not cover such kernels, as $K_\alpha(0+)=\infty$. Moreover, weak convergence results approximating the solution of \eqref{VolterraCIR} for $K=K_\alpha$ with a sequence of solutions corresponding to nonexplosive kernels (see, e.g., \cite[Theorem 2.8]{ABI_j} and \cite[Theorem 3.5]{ABE1}) do not allow to extend the conditions in Theorem \ref{Fel_CIR} to the fractional limiting case.\vspace{2mm}

Nevertheless, Theorem \ref{Fel_CIR} can be applied to approximations of $K_\alpha$ which are commonly used in applications. In particular, we here focus on
\begin{equation}\label{ex_fr}
K_{\alpha,T}(t)=\int_{0}^{T}e^{-xt}\mu(\dd x),\,t\ge0, \quad \text{where}\quad \mu(\dd x)=\frac{x^{-\alpha}}{\Gamma(\alpha)\Gamma(1-\alpha)}\dd x\eqqcolon w(x)\,\dd x ,
\end{equation}
for some $T>0$. Since $K_{\alpha,T}$ belongs to $C^1(\mathbb{R}_+)$ and is completely monotone on $(0,\infty)$, it satisfies Hypotheses~\ref{Hyp1}-\ref{hyp_increasing}-\ref{hyp_res}. We denote by $X_{\alpha,T}$ the pathwise unique strong solution of   \eqref{VolterraCIR} associated with this kernel.  Note that $K_{\alpha,T}$ -- considered, for instance, in \cite{AK} -- is obtained by truncating the integral in the Bernstein-Widder representation of $K_\alpha$, that is, $K_{\alpha,T}(t)=\int_{0}^{\infty} e^{-xt}\mu(\dd x),\,t>0$. A direct computation entails that 
\[
	K_{\alpha,T}(0)=\frac{T^{1-\alpha}}{\Gamma(\alpha)\Gamma(2-\alpha)},\qquad 
	K'_{\alpha,T}(0)=-\frac{T^{2-\alpha}}{(2-\alpha)\Gamma(\alpha)\Gamma(1-\alpha)}.
\]
Given an input  $(x_0,\kappa,\theta,\sigma)$ for \eqref{VolterraCIR}, since $K_{\alpha,T}(0)\to\infty$ as $T\to\infty$, the sufficient condition in Theorem \ref{Fel_CIR} \ref{suff_CIR}  is not verified for $T$ sufficiently large. This might suggest that, for $T$ large enough, $X_{\alpha,T}$ attains the value $0$.  To rigorously prove this intuition, it suffices to show that the necessary condition in  Theorem \ref{Fel_CIR} \ref{nec_CIR} is violated. We then compute, denoting by $C_1,C_2$ two positive constants possibly depending on $\alpha, \kappa,\theta$ and $\sigma$, 
\[
\sigma^2\frac{K_{\alpha,T}(0)^3}{2|K_{\alpha,T}'(0)|}-\kappa\theta\frac{K_{\alpha,T}(0)^2}{|K_{\alpha,T}'(0)|}
=C_1 \frac{1}{T^{2\alpha-1}}-C_2\frac{1}{T^\alpha}\underset{T\to \infty}{\longrightarrow}\infty \quad \text{if and  only if}\quad \alpha\in \Big(0,\frac{1}{2}\Big).
\]
As a result, the previous conjecture is only proven for $\alpha\in(0,\frac{1}{2})$, namely, for these values of $\alpha$,  the process $X_{\alpha,T}$ reaches $0$ with positive probability for every initial condition $x_0>0$, when $T$ is sufficiently large. \vspace{2mm}

The integral expression of $K_{\alpha,T}=\int_{0}^{T}e^{-x\cdot}\mu(\dd x)$ is at the basis of multi-factor Markovian approximations of rough volatility models in the spirit of \cite{edu-Lift, ABE1}. We also refer to \cite{BayerBreneis, BB2} for more efficient implementations of the same method. In these approaches, the idea is to consider an increasing (finite) sequence $(\xi_{n,N})_{n=0,\dots,N}\subset\mathbb{R}_+$, for some $N\in \mathbb{N}$, and then to approximate  -- on the resulting intervals $[\xi_{n-1,N}, \xi_{n,N}]$, $n=1,\dots,N$ -- $\mu(\dd x)$ in \eqref{ex_fr}  with a weighted sum of $q$ Dirac measures, with $q\in \mathbb{N}$. Consequently, the approximated kernels $K_{\alpha,Nq}$ have the form 
\[
	K_{\alpha,Nq}(t)=\sum_{n=1}^{Nq}m_{n,Nq}\,e^{-x_{n,Nq}t},\quad t\ge0,\text{ for some }(m_{n,Nq})_n,(x_{n,Nq})_n\in (0,\infty)^{Nq}.
\]
We denote by  $X_{\alpha,Nq}$ the solution process of \eqref{VolterraCIR} corresponding to $K_{\alpha,Nq}\in C^1(\mathbb{R}_+)$, which is completely monotone on $(0,\infty)$ because it is a weighted sum of exponentials.
If, as is the case for \cite{edu-Lift,ABE1,BayerBreneis}, the masses $(m_{n,Nq})_n$ and the centers of mass  $(x_{n,Nq})_n$ are selected according to Gaussian quadrature rules of order $q$ with the weight function $w(x)$ in \eqref{ex_fr}, then (see e.g. \cite{GW}) 
\begin{equation}
\begin{aligned}\label{eq:K_alphan}
	&K_{\alpha,Nq}(0)=\sum_{n=1}^{Nq}m_{n,Nq}=\mu([\xi_{0,N},\xi_{N,N}])=
	\frac{\xi_{N,N}^{1-\alpha}-\xi^{1-\alpha}_{0,N}}{\Gamma(\alpha)\Gamma(2-\alpha)},\\
	&K'_{\alpha,Nq}(0)=-\sum_{n=1}^{Nq}m_{n,Nq}x_{n,Nq}=-\int_{\xi_{0,N}}^{\xi_{N,N}}x\mu(\dd x)=-
	\frac{\xi_{N,N}^{2-\alpha}-\xi^{2-\alpha}_{0,N}}{(2-\alpha)\Gamma(\alpha)\Gamma(1-\alpha)}.
\end{aligned} 
\end{equation}
It is customarily supposed that the sequence $(\xi_{N,N})_N$ [resp., $(\xi_{0,N})_N$] diverges to $\infty$ [resp., is bounded], which implies that $K_{\alpha,Nq}(0)\to \infty$ as $N\to \infty$. Therefore, Theorem \ref{Fel_CIR} \ref{suff_CIR} is not satisfied for $N$ big enough.  As for the necessary condition in Theorem \ref{Fel_CIR} \ref{nec_CIR}, we note that, for a constant $C_3=C_3 (\alpha, \sigma)>0$,
\[
	\sigma^2\frac{K_{\alpha,Nq}(0)^3}{2|K_{\alpha,Nq}'(0)|}-\kappa\theta\frac{K_{\alpha,Nq}(0)^2}{|K_{\alpha,Nq}'(0)|}
		\sim C_3\frac{1}{\xi_{N,N}^{2\alpha-1}}
	\quad 
	\text{ as }\quad N\to\infty.
\]
Thus, when $\alpha\in (0,\frac{1}{2})$, analogously to $X_{\alpha,T}$, the process $X_{\alpha,Nq}$  attains the value $0$ for every initial condition when $N$ is sufficiently large. This case corresponds to the multi-factor approximation of hyper-rough Volterra Heston models, see \cite[Section 7]{ABI_j}. 
The same general conclusion cannot be drawn for $\alpha\in (\frac{1}{2},1)$, when $\lim_{N\to\infty}\xi^{1-2\alpha}_{N,N}=0$ and all initial states $x_0>0$ satisfy Theorem \ref{Fel_CIR} \ref{nec_CIR} for $N$ large enough. In most applications, however, a low number of factors $Nq$ is needed for improved efficiency. Therefore, if $\alpha$ is close to its lower bound $\frac{1}{2}$, one might still expect such necessary condition to be violated  for a wide range of starting points, owing to the slow decay of $\xi_{N,N}^{1-2\alpha}$. 
\vspace{2mm}\\
We finally comment on the geometric and non-geometric Gaussian schemes in the recent work \cite{BB2}. These methods are preferable because they offer faster numerical performance, which is useful, for example, when calibrating models involving $X_{\alpha,Nq}$ to  real market data, see \cite[Section 6]{rhh}. In these approaches, $\xi_{0,N}=0$ for all $N$, $(\xi_{1,N})_N$ is bounded, $(\xi_{N,N})_N$ diverges and the Gaussian quadrature rules of order $q$ used to determine $(m_{n,Nq})_n$ and $(x_{n,Nq})_n$  employ the weight function $w(x)\equiv 1$ (apart from $[0, \xi_{1,N}]$,  where $w(x)$ is as in \eqref{ex_fr}). As a result, similarly to \eqref{eq:K_alphan},
\begin{equation*}
	K_{\alpha,Nq}(0)=
	\frac{\xi_{1,N}^{1-\alpha}}{\Gamma(\alpha)\Gamma(2-\alpha)}+(\xi_{N,N}-\xi_{1,N}),\qquad 
	K'_{\alpha,Nq}(0)=-
	\frac{\xi_{1,N}^{2-\alpha}}{(2-\alpha)\Gamma(\alpha)\Gamma(1-\alpha)}-\frac{1}{2}(\xi^2_{N,N}-\xi^2_{1,N}).
\end{equation*}
In particular, $$K_{\alpha,Nq}(0)\to \infty \quad \text{ and }\quad \sigma^2\frac{K_{\alpha,Nq}(0)^3}{2|K_{\alpha,Nq}'(0)|}-\kappa\theta\frac{K_{\alpha,Nq}(0)^2}{|K_{\alpha,Nq}'(0)|}\to \infty\quad  \text{ as } \quad N\to \infty.$$
As regards Theorem \ref{Fel_CIR}, these computations entail that, for all $\alpha\in (0,1)$ and $x_0>0$, both the sufficient condition in Point \ref{suff_CIR} and the necessary condition in Point  \ref{nec_CIR} are violated for $N$ large enough. In fact, we have verified that $X_{\alpha,Nq}$ attains $0$ in a nonnegligible event in the two following literature cases, where $Nq$ is low for numerical efficiency.
\begin{itemize}
	\item \cite{BB2}, both using the geometric Gaussian scheme as specified in \cite[Theorem 3.9]{BB2} and the non-geometric Gaussian scheme as specified in  \cite[Theorem 3.16]{BB2}, with $a=6.4$. Here $\alpha\in\{0.4,0.5,\dots, 0.9\}$, $Nq=4$ and the test parameters are in \cite[Subsection 4.2]{BB2}; 
	\item \cite[Subsection 6.1]{rhh}, where an $L^1-$optimized version of the geometric Gaussian scheme in \cite{BB2} is considered. Here  $Nq=20$ and the parameters set comes from a joint calibration of the rough Heston model to SPX and VIX market data.
\end{itemize}
\subsection{The Volterra Jacobi process}\label{subsection_JV}
  Fix $a,b\in \mathbb{R}$ such that $a< b$, $\kappa>0$, $\sigma>0$, and let $\theta\in (a,b)$. Define the nonnegative quadratic function $Q_{a,b}\colon[a,b]\to\mathbb{R}_+$ by $Q_{a,b}(x)=(x-a)(b-x),\,x\in[a,b]$. Consider the SVE 
 \begin{equation}\label{VJac}
 	X_t=x_0+\int_{0}^{t}K(t-s)\kappa(\theta-X_s)\,\dd s+ \sigma \int_{0}^{t}K(t-s)\sqrt{Q_{a,b}(X_s)}\,\dd W_s,\quad  t\ge 0, \,\mathbb{P}-\text{a.s.},
 \end{equation}
where $x_0\in (a,b)$. By Hypotheses \ref{Hyp1}-\ref{hyp_res}, for every finite time horizon $N\in \mathbb{N}$, the existence of a  continuous weak solution to \eqref{VJac}  in $[0,N]$ is established in \cite[Corollary 2.8]{Pol_sergio}. Note that, for every $x,y\in [a,b],$ by the subadditivity of the square-root function, 
\[
	\left|\sqrt{Q_{a,b}(x)}-\sqrt{Q_{a,b}(y)}\right|
	\le 
	\sqrt{\left|Q_{a,b}(x)-Q_{a,b}(y)\right|}\frac{\sqrt{Q_{a,b}(x)+Q_{a,b}(y)}}{\sqrt{Q_{a,b}(x)}+\sqrt{Q_{a,b}(y)}}
	\le \sqrt{3(|a|+|b|)}\sqrt{|x-y|},
\]
which gives the $\frac{1}{2}-$H\"older's continuity of $\sqrt{Q_{a,b}}$ in $[a,b]$. Therefore, by the Yamada-Watanabe type argument in \cite[Proposition B.3]{ABE1}, we infer that there exists a pathwise unique continuous strong solution $X^N=(X^N_t)_{t\in[0,N]}$ of \eqref{VJac} in $[0,N]$. If we define the process $X=(X_t)_{t\ge 0}$ by
\[
X_t=X^{[t]+1}_t,\quad t \ge0,
\]
then $X$ is the pathwise unique, $[a,b]-$valued continuous strong solution of \eqref{VJac} in $\mathbb{R}_+$. We call such $X$ a \emph{Volterra Jacobi process}. 
As in Subsection \ref{subsection_CIR}, we now want to apply the results in Section \ref{main} to determine conditions on the parameters $\kappa,\,\theta,\,\sigma$ and the initial point $x_0$ such that the paths of $X$ touch  (or do not touch) the boundary of $[a,b].$

\vspace{2mm}
Fix  $l=a$, $r=b$ and define the modified drift and diffusion mappings $\tilde{b},\,\tilde{\sigma}\colon (a,b)\to \mathbb{R}$ by
\[
\tilde{b}(x)=K(0)\kappa(\theta-x)+\frac{K'(0)}{K(0)}x,\qquad \tilde{\sigma}(x)=K(0)\sigma\sqrt{Q_{a,b}(x)},\quad x\in (a,b).
\]
For an arbitrary $c\in(a,b)$, given $\beta,\,\gamma\in\mathbb{R}$, by \eqref{def_bc}-\eqref{scale_func_def},  the derivative of the scale function $p_c(\cdot;\beta,\gamma)\colon(a,b)\to \mathbb{R}$ associated with the dynamics \eqref{VJac} in the interval $(a,c)$ is  
\begin{align}\label{p1_J-}
	\notag	p'_c(x;\beta)&=\exp\left\{ -\frac{2}{K(0)^2\sigma^2}\int_{c}^{x}\frac{1}{Q_{a,b}(y)}\left(K(0)\kappa(\theta-y)+\frac{K'(0)}{K(0)}(y+\beta)\right) \dd y\right\}
	\\&\notag=
	\exp\left\{-\frac{2}{K(0)^2\sigma^2}\int_{c}^{x}\left(\frac{A(\beta)}{y-a}+\frac{B(\beta)}{b-y}\right)\dd y\right\}
	\\&=
	\left[\left(\frac{x-a}{c-a}\right)^{A(\beta)}\left(\frac{b-x}{b-c}\right)^{-{B(\beta)}}\right]^{-{2}(K(0)\sigma)^{-2}}
	,\quad x\in (a,c), 
\end{align}
where
\begin{equation}\label{defAB}
A(\beta)=\frac{1}{b-a}\left(K(0)\kappa(\theta-a)+\frac{K'(0)}{K(0)}(a+\beta)\right),\quad B(\beta)=\frac{1}{b-a}\left(K(0)\kappa(\theta-b)+\frac{K'(0)}{K(0)}(b+\beta)\right).
\end{equation}
Analogously, for every $x\in [c,b)$, 
\begin{align}\label{p1_J+}
	p'_c(x;\gamma)=	\left[\left(\frac{x-a}{c-a}\right)^{A(\gamma)}\left(\frac{b-x}{b-c}\right)^{-{B(\gamma)}}\right]^{-{2}(K(0)\sigma)^{-2}},
\end{align}
with $A(\gamma)$ and $B(\gamma)$ defined as in \eqref{defAB} replacing $\beta$ by $\gamma$. \\
We now focus on the  nonnegative map $v_c(\cdot;\beta,\gamma) \colon (a,b)\to\mathbb{R}_+$  in \eqref{v_def2}. Defining $C=2(K(0)\sigma)^{-2}$ and supposing that $A(\beta)\neq0$,  the following lower bound for $v_c(\cdot;\beta)$ holds in the interval $(a,c)$:
\begin{align}\label{lower_b_JV}
	v_c(x;\beta)&\notag=C{(c-a)^{-A(\beta)C}
		(b-c)^{B(\beta)C}
	}\int_{x}^{c}p'_c(y;\beta)\left(\int_{y}^{c}	
	(z-a)^{A(\beta)C-1}(b-z)^{-B(\beta)C-1}
	\dd z \right)\dd y\\&\notag
	\ge 
	C_1
	\left(A(\beta)C\right)^{-1	}
	\int_{x}^{c}p'_c(y;\beta)
	\left((c-a)^{A(\beta)C}-(y-a)^{A(\beta)C}\right)
	\dd y
	\\
	&
	\ge C_1\int_{x}^{c}
	\bigg|\left(\frac{y-a}{c-a}\right)^{-A(\beta)C}
	-1\bigg|\,\dd y
	,\quad x\in (a,c).
\end{align}
We can also establish an analogous upper bound on the same interval, given by
\begin{equation}\label{upper_b_JV}
	 v_c(x;\beta)\le C_2
\int_{x}^{c}
	\bigg|\left(\frac{y-a}{c-a}\right)^{-A(\beta)C}
	-1\bigg|\,\dd y
	,\quad x\in (a,c).
\end{equation}
Supposing that also $B(\gamma)\neq 0$, similar computations enable us to estimate $v_c(\cdot;\gamma)$ in the interval $[c,b)$, namely
\begin{equation}\label{est_[c,b]}
		 (-1)^iv_c(x;\gamma)\le (-1)^iC_i
	\int_{c}^{x}
	\bigg|\left(\frac{b-y}{b-c}\right)^{B(\gamma)C}
	-1\bigg|\,\dd y
	,\quad x\in [c,b),\,i=3,4.
\end{equation}
Here $C_1,\,C_2,\,C_3,\,C_4>0$ are  positive constants, possibly depending on $K,\,\kappa,\,\theta,\,\sigma,\,\beta,\,\gamma,\,a,\,b$ and $c$, which are  allowed to change from line to line. 
Combining the previous estimates  with the results in Subsection \ref{subsec_v}, we obtain the next theorem. 
\begin{theorem}\label{Fel_J}
Assume Hypotheses \ref{Hyp1}-\ref{hyp_increasing}-\ref{hyp_res}. Consider  the stopping times $S_-=\inf\{t\ge 0:X_t=a\}$, $S_+=\inf\{t\ge0 : X_t=b\}$ and $S=S_-\wedge S_+=\inf\{t\ge 0 : X_t\in \{a,b\}\}$. 
\begin{enumerate} [label=(\roman*)]
	\item \label{nec_J} Suppose that $K'(0)<0$. Then, the condition $$\text{$x_0\ge a+\frac{K(0)^2}{2|K'(0)|}\left(K(0)\sigma^2(b-a)-2\kappa(\theta-a)\right)$ \quad  is necessary for\quad  $\mathbb{P}(S_-=\infty)=1$.}$$
	 Additionally, the condition $$\text{$x_0\le 	b-\frac{K(0)^2}{2|K'(0)|}\left(K(0)\sigma^2(b-a)-2\kappa(b-\theta)\right)		$\quad  is necessary for \quad $\mathbb{P}(S_+=\infty)=1$}.$$
	\item \label{suff_J} The condition $2\kappa{(\theta-a)}\ge K(0)\sigma^2(b-a)$ is sufficient for  $\mathbb{P}(\{S<S_-\}\cup\{S_-=\infty\} )=1$. Additionally, the condition $2\kappa{(b-\theta)}\ge K(0)\sigma^2(b-a)$ is sufficient for $\mathbb{P}(\{S<S_+\}\cup\{S_+=\infty\} )=1$. Consequently, 
	\[
	2\kappa\min\{\theta-a,b-\theta\}\ge K(0)\sigma^2(b-a)\quad \text{ implies that }\quad  \mathbb{P}(S=\infty)=1.
	\]
\end{enumerate}
\end{theorem}
\begin{proof}
	Let $c\in (a,b).$ We will only prove the assertions concerning $S_-$, as those about $S_+$ can be obtained by analogous arguments relying on the estimate \eqref{est_[c,b]}. We first demonstrate the necessary condition in Point~\ref{nec_J} by contradiction. Consider $x_0<a+{K(0)^2}(2|K'(0)|)^{-1}\left(K(0)\sigma^2(b-a)-2\kappa(\theta-a)\right)$, which, recalling \eqref{defAB}, amounts to supposing that
	$$
	A(-x_0)C=2\big(K(0)\sigma\big)^{-2}\frac{1}{b-a}\left(K(0)\kappa(\theta-a)+\frac{K'(0)}{K(0)}(a-x_0)\right)<1
	.$$
	%
	We can then select a  $\bar{\beta}>x_0>a$ such that $A(-\bar{\beta})\neq 0$ and  $\bar{\beta}-x_0>0$ is so small that the previous bound holds replacing $\bar{\beta}$ for $x_0$, i.e., $A(-\bar{\beta})C<1$.
	In this way, by \eqref{upper_b_JV}, for  a positive constant $C_{\bar{\beta}}$,
	\[
	v_c(a+;-\bar{\beta})\le C_{\bar{\beta}}
\int_{a}^{c}
	\bigg|\left(\frac{y-a}{c-a}\right)^{-A(\bar{\beta})C}
	-1\bigg|\dd y<\infty.
	\]
 Thus,  the necessary condition in    Theorem \ref{Feller_thm_nec} \ref{nec_lower} is violated, which yields $\mathbb{P}(S_-<\infty)>0$. This is absurd, proving Point \ref{nec_J}.

	As for Point \ref{suff_CIR}, we assume $2\kappa(\theta-a) \ge K(0)\sigma^2(b-a)$ and observe that, by   \eqref{lower_b_JV} and the definition of $A(-a)$ in  \eqref{defAB}, 
	\[
	v_c(a+;-a)\ge {\widetilde{C}}\int_{a}^{c}
	\bigg(\left(\frac{ c-a}{y-a}\right)^{2\kappa(\theta-a){(K(0)\sigma^2(b-a))^{-1}}}-1\bigg)\dd y=\infty,\quad \text{for some }{\widetilde{C}}>0.
	\]
	Therefore, thanks to Remark \ref{rem_cl_suff},  we invoke Theorem \ref{Feller_thm} \ref{suff_lower} to conclude that  $\mathbb{P}(\{S<S_-\}\cup\{S_-=\infty\} )=1$. The theorem is now completely proved. 
\end{proof}
In the Volterra Jacobi framework the interval $I=(a,b)$ is bounded. Thus, we can employ Proposition \ref{suff_thm} to obtain the following result.
\begin{proposition}\label{J_locsup}
	Assume Hypotheses \ref{Hyp1}-\ref{hyp_increasing}-\ref{hyp_res}. Define $S_+,\,S_-$ and $S$ as in Theorem \ref{Fel_J}.
	\begin{enumerate}[label=(\roman*)]
		\item \label{jv_suff}If $2\kappa(b-\theta)\ge K(0)\sigma^2(b-a)$ and $2\kappa(\theta-a)<\left(K(0)\sigma^2-2|K'(0)|K(0)^{-2}\right)(b-a)$, then $$\mathbb{P}\left(\sup_{0\le t< S} X_t<b\right)=1;$$
		\item \label{jvdua_suff}
If $2\kappa(\theta-a)\ge K(0)\sigma^2(b-a)$ and $2\kappa(b-\theta)<\left(K(0)\sigma^2-2|K'(0)|K(0)^{-2}\right)(b-a)$, then 
$$\mathbb{P}\left(\inf_{0\le t< S} X_t>a\right)=1.$$
	\end{enumerate}
\end{proposition}
\begin{proof}
	We focus only on Point \ref{jv_suff}, as Point \ref{jvdua_suff} can be deduced by similar arguments in a straightforward way. Note that, by \eqref{defAB}, the condition $2\kappa(\theta-a)<\left(K(0)\sigma^2-2|K'(0)|K(0)^{-2}\right)(b-a)$ is equivalent to 
	\[
	A(-b)C=\frac{2}{K(0)\sigma^2(b-a)}\left(\kappa(\theta-a)+\frac{K'(0)}{K(0)^2}(a-b)\right)<1.
	\]
	Furthermore,  
	\[
	B(-b)C=\frac{2\kappa(\theta-b)}{K(0)\sigma^2(b-a)}\le -1\quad \Longleftrightarrow \quad 2\kappa(b-\theta)\ge K(0)\sigma^2(b-a).
	\]
By \eqref{p1_J-}-\eqref{p1_J+}, the previous estimates show that $p_c(a+;-b)>-\infty$ and $p_c(b-;-b)=\infty$, for some $c\in(a,b)$.	
	Hence an application of Proposition \ref{suff_thm} \ref{p1_suff} yields the desired conclusion, and the proof is complete. 
\end{proof}
\begin{rem}\label{classical_J}
	When $K\equiv K(0)$ in $\mathbb{R}_+$, for some positive constant $K(0)>0$, we can obtain a necessary and sufficient condition for $\mathbb{P}(S=\infty)=1$. More precisely, the following equivalence holds: 
	\begin{equation}\label{fel_con_J}
		\text{$\mathbb{P}(X_t\in (a,b),\,t\in\mathbb{R}_+)=1\quad $ if and only if $\quad 2\kappa \min\{\theta-a,b-\theta\}\ge K(0)\sigma^2(b-a)$.}
	\end{equation}
	Indeed, the fact that  $2\kappa  \min\{\theta-a,b-\theta\}\ge K(0)\sigma^2(b-a) \,\Longrightarrow\, \mathbb{P}(X_t\in (a,b),\,t\in\mathbb{R}_+)=1$ is given by Theorem \ref{Fel_J} \ref{suff_J}.  The reverse implication is instead deduced by \eqref{upper_b_JV}-\eqref{est_[c,b]}, Corollary \ref{coro_nec_suff} and Remark \ref{classical_case}. Notice that, choosing $K\equiv1$, \eqref{fel_con_J} reduces to boundary attainment results for (classical) Jacobi processes already established in the literature, see for instance \cite[Theorem 2.1]{AFP}. 
	
	In fact, when the condition in \eqref{fel_con_J} is not met, we can infer more properties on the trajectories of $X$. Specifically, if $2\kappa(\theta-a)<K(0)\sigma^2(b-a)$, then by \eqref{upper_b_JV}, Corollary \ref{coro_nec_suff} \ref{suffnec_lo} and Remark \ref{classical_case}, the paths of $X$ attain the value $a$ with positive probability. If we further suppose that $2\kappa(b-\theta)\ge K(0)\sigma^2(b-a)$, then $\mathbb{P}(\sup_{t\in [0,S)}X_t<b)=1$ by Proposition \ref{J_locsup}\ref{jv_suff}.	On the other hand, if $2\kappa(b-\theta)<K(0)\sigma^2(b-a)$, then by \eqref{est_[c,b]}, Corollary \ref{coro_nec_suff} \ref{suffnec_up} and Remark \ref{classical_case}, the trajectories of $X$ attain the value $b$ with positive probability. 
	Furthermore, if also $2\kappa(\theta-a)\ge K(0)\sigma^2(b-a)$, then $\mathbb{P}(\inf_{ t\in[0, S)}X_t>a)=1$ by Proposition \ref{J_locsup}\ref{jvdua_suff}.
\end{rem}

\subsection{The  Volterra power-type diffusion process}\label{subsection_power}
	Given $\sigma>0$ describing the strength of the random perturbation and coefficients $\alpha>1$ and $\delta\in [0,1)$ characterizing the power-type drift and diffusion, respectively,  we consider the SVE with multiplicative noise 
	 \begin{equation}\label{V_power}
	 	X_t=x_0+\int_{0}^{t}K(t-s)|X_s|^\alpha\dd s+ \sigma \int_{0}^{t}K(t-s)|X_s|^{\frac{\delta}{2}}\,\dd W_s,\quad  t\ge 0, \,\mathbb{P}-\text{a.s.},
	 \end{equation}
 where  $x_0\in \mathbb{R}$. We suppose that there exists a continuous $\widebar{\mathbb{R}}-$valued process $X$ satisfying \eqref{V_power} locally in the sense of \eqref{SDE_VolterraC1}. In this framework,  boundary attainment implies  \emph{blow-ups} (also called \emph{explosions}) to $\pm \infty$ of  the trajectories. Our aim in this subsection is  to determine a relation between the exponents $\alpha$ and $\delta$ ensuring that $X$ explodes in finite time,  see Theorem \ref{Explosion_thm} below.  Note that, in the limiting case $\alpha=1$, the drift reduces to a map with linear growth, hence the existence of a global weak solution to \eqref{V_power} is given by \cite[Theorem 3.4]{sergio}, for every $\delta\in [0,2]$. We also refer to \cite[Theorem 3.7]{Zhang}, which establishes the existence and uniqueness of a maximal (local) solution to \eqref{V_power}. Such result requires locally Lipschitz continuous coefficients, hence it can be applied to our setting only when $\delta=0$. 
 \\
 Fix  $l=-\infty$, $r=\infty$ and define the mappings $\tilde{b},\,\tilde{\sigma}\colon\mathbb{R}\to \mathbb{R}$ by
 \[
 \tilde{b}(x)=K(0)|x|^\alpha+\frac{K'(0)}{K(0)}x,\qquad \tilde{\sigma}(x)=\sigma K(0)|x|^\frac{\delta}{2},\quad x\in \mathbb{R}.
 \]
 We observe that the modified drift $\tilde{b}$ and diffusion $\tilde{\sigma}$ satisfy Hypothesis \ref{hyp_coeff} because $\delta<1$. \\
 To ease the notation, we let $C=2(K(0)\sigma)^{-2}$. For $\beta,\,\gamma\in\mathbb{R}$ and $c=0$,  the derivative of the scale function $p_0(\cdot;\beta,\gamma)\colon\mathbb{R}\to \mathbb{R}$ is, by \eqref{def_bc}-\eqref{scale_func_def},  recalling also that $K'(0)\le 0,$
 \begin{multline}\label{p0_power-}
		p'_0(x;\beta)=\exp\bigg\{ -C
	\bigg(
		-\frac{K(0)}{\alpha-\delta+1}(-x)^{\alpha-\delta+1}-\frac{|K'(0)|}{(2-\delta)K(0)}(-x)^{2-\delta}\\+\beta\frac{|K'(0)|}{(1-\delta)K(0)}(-x)^{1-\delta}
	\bigg)
	\bigg\}
	,\quad x<0,
\end{multline}
and 
\begin{equation}\label{p0_power+}
		p'_0(x;\gamma)=\exp\left\{ -C
		\left(
		\frac{K(0)}{\alpha-\delta+1}x^{\alpha-\delta+1}-\frac{|K'(0)|}{(2-\delta)K(0)}x^{2-\delta}-\gamma\frac{|K'(0)|}{(1-\delta)K(0)}x^{1-\delta}
		\right)
	\right\}
	,\quad x\ge0.
\end{equation}
Observe that we can estimate from below the map $v_0(\cdot;\beta) \colon (-\infty,-1)\to\mathbb{R}_+$, defined   in \eqref{v_def}-\eqref{v_def2}, as follows:
\begin{equation*}
		v_0(x;\beta)
		\ge C\int_{x}^{-1}p_0'(y;\beta)\left(\int_{y}^{0}\frac{\dd z}{(-z)^\delta p_0'(z;\beta)}\right)\dd y
		\ge C \left(\int_{-1}^{0}\frac{\dd z}{(-z)^\delta p_0'(z;\beta)}\right)
		\left(\int_{x}^{-1}p_0'(y;\beta)\dd y\right),\quad x<-1.
\end{equation*}
In particular, when $\beta>0$, by \eqref{p0_power-},
\begin{align}\label{lower_bo_p}
	\notag	v_0(x;\beta)\ge &C \left(\int_{-1}^{0}\frac{\dd z}{(-z)^\delta p_0'(z;\beta)}\right)\exp\left\{-\frac{C\beta|K'(0)|}{(1-\delta)K(0)}(-x)^{1-\delta}\right\}\\&\times 
\int_{x}^{-1}\exp\left\{ C
\left(
\frac{K(0)}{\alpha-\delta+1}(-y)^{\alpha-\delta+1}+\frac{|K'(0)|}{(2-\delta)K(0)}(-y)^{2-\delta}
\right)
\right\}\dd y
,\quad x<-1.
\end{align}
Combining this estimate with further assumptions on $\alpha$ and $\delta$,  which enable us to obtain un upper bound for $v_0(\cdot;\gamma)$ in a positive  half-line, we infer that the solution $X$ of \eqref{V_power} blows up in finite time with positive probability. More specifically, the next result holds.
\begin{theorem}\label{Explosion_thm}
	Assume Hypotheses \ref{Hyp1}-\ref{hyp_increasing}. Consider a continuous $\widebar{\mathbb R}-$valued process $X$ solving \eqref{V_power} in the sense of  \eqref{SDE_VolterraC1}. Define the corresponding stopping times $S_-=\inf\{t\ge 0:X_t=-\infty\}$, $S_+=\inf\{t\ge0 : X_t=\infty\}$ and $S=S_-\wedge S_+=\inf\{t\ge 0 : X_t\in \{-\infty,\infty\}\}$. Then 
	$$\mathbb{P}(\{S<S_-\}\cup\{S_-=\infty\} )=1.$$ 
	Furthermore, if $\alpha>1+\delta$, then $\mathbb{P}(S_+<\infty)>0$.
\end{theorem}
\begin{proof}
	We first verify the sufficient condition in Theorem \ref{Feller_thm} \ref{suff_lower}, which guarantees that $\mathbb{P}(\{S<S_-\}\cup\{S_-=\infty\} )=1$, as desired. We consider $l_n=-n,\,n\in \mathbb{N},$ and aim to show that 
	\begin{equation}\label{obj_1part}
		\lim_{n\to \infty} v_0(-n;n)=\infty.
	\end{equation}
Thanks to the lower bound in \eqref{lower_bo_p}, for all $x<-1$, 
	\begin{align*}
			\notag	v_0(x;-x)\ge &C \left(\int_{-1}^{0}\frac{\dd z}{(-z)^\delta p_0'(z;-x)}\right)\exp\left\{-\frac{C|K'(0)|}{(1-\delta)K(0)}(-x)^{2-\delta}\right\}\\&\times 
		\int_{x}^{-1}\exp\left\{ C
		\left(
		\frac{K(0)}{\alpha-\delta+1}(-y)^{\alpha-\delta+1}+\frac{|K'(0)|}{(2-\delta)K(0)}(-y)^{2-\delta}
		\right)
		\right\}\dd y\eqqcolon C
		(\mathbf{\upperRomannumeral{1}}\times \mathbf{\upperRomannumeral{2}}^{-1}\times \mathbf{\upperRomannumeral{3}})(x).
	\end{align*}
	Note that, by \eqref{p0_power-}, $p_0'(z;-x)\le p_0'(z;0)$ for all $z<0$. Thus, 
	$$\mathbf{\upperRomannumeral{1}}(x)\ge \int_{-1}^{0}\frac{\dd z}{(-z)^\delta p_0'(z;0)}, \quad  x<-1.$$ 
	Next, denoting by ${C}_i>0,\,i=1,\dots, 4,$  positive constants possibly depending on $C,\,K,\,\alpha$ and $\delta$, since $\alpha>1$ we have
	\begin{align*}
		\lim_{x\to-\infty }\frac{\mathbf{\upperRomannumeral{3}}'(x)}{\mathbf{\upperRomannumeral{2}}'(x)}
		=
				\lim_{x\to-\infty }\frac
				{\exp\left\{ C_1
				(-x)^{\alpha-\delta+1}+C_2(-x)^{2-\delta}
					\right\}}
				{
					C_4(-x)^{1-\delta}\exp\left\{C_3(-x)^{2-\delta}\right\}
				}=\infty.
	\end{align*}
	Therefore, $\lim_{x\to -\infty}\mathbf{\upperRomannumeral{2}}^{-1}(x)\times \mathbf{\upperRomannumeral{3}}(x)=\infty$  by De l'H\^opital's rule, whence \eqref{obj_1part}.
	
	Secondly, we assume $\alpha>1+\delta$ and argue that the necessary condition in Theorem \ref{Feller_thm_nec} \ref{nec_upper} is violated. More precisely, we are going to show that $v_0(\infty-; \gamma)<\infty$ for all $\gamma\in \mathbb{R}$. We denote by $g(x),\,x\ge 0,$ the exponent of $p_0'(\cdot;\gamma)$ in \eqref{p0_power+} divided by $-C$, so that we can write
	\begin{equation}\label{trick_proof}
		p_0'(x;\gamma)=\exp\left\{
		-C  g(x)
		\right\},\quad x\ge 0.
	\end{equation}
	Since $g'\sim K(0)x^{\alpha-\delta}$ and $g''\sim {K(0)}(\alpha-\delta)x^{\alpha-\delta-1}$ as $x\to \infty$, there exists a threshold $\tilde{c}=\tilde{c}(\alpha,\gamma,\delta, K)>1$ such that $g'$ is increasing and strictly positive in the half-line $[\tilde{c},\infty)$. As a result, we can deduce an upper bound for $v_0(\cdot;\gamma)$ in $(\tilde{c},\infty)$ using the following decomposition: 
	\begin{align}\label{upper_b_power}
	\notag	v_0(x;\gamma)\le &C\left(\int_{0}^{\tilde{c}}p_0'(y;\gamma)\dd y\right)\left(\int_{0}^{\tilde{c}}\frac{\dd z}{z^\delta p_0'(z;\gamma)}\right)
		+C
	\left(\int_{0}^{\tilde{c}}\frac{\dd z}{z^\delta p_0'(z;\gamma)}\right)\left(	\int_{\tilde{c}}^{x}p_0'(y;\gamma)\dd y\right)
		\\&\quad +
	C	\int_{\tilde{c}}^{x}p_0'(y;\gamma)\left(\int_{\tilde{c}}^y\frac{\dd z}{z^\delta p_0'(z;\gamma)}\right)\dd y
		,\quad x>\tilde{c}.
	\end{align}
In particular, if we focus on the last addend in the right-hand side of \eqref{upper_b_power}, then, by Tonelli's theorem, \eqref{trick_proof} and the fact that $\tilde{c}>1$, 
\begin{align*}
		\int_{\tilde{c}}^{x}p_0'(y;\gamma)\left(\int_{\tilde{c}}^y\frac{\dd z}{z^\delta p_0'(z;\gamma)}\right)\dd y&\le 
			\int_{\tilde{c}}^{\infty }\frac{1}{p_0'(z;\gamma)}\left(\int_{{z}}^\infty \frac{g'(y)}{g'(y)}p_0'(y;\gamma)\dd y\right)\dd z\\
			&\le \frac{1}{C} \int_{\tilde{c}}^{\infty}\frac{1}{g'(z)}\dd z,\quad x>\tilde{c}.
\end{align*}
Plugging this estimate in \eqref{upper_b_power}, we obtain, for every $x>\tilde{c},$
\[
	v_0(x;\gamma)\le C \left(\int_{0}^{\tilde{c}}\frac{C\dd z}{z^\delta p_0'(z;\gamma)}\right)
	\left(\int_{0}^{\infty}p_0'(y;\gamma)\dd y\right)+
	 \int_{\tilde{c}}^{\infty}\frac{1}{g'(z)}\dd z
	<\infty.
\]
	Here, the two addends on the right-hand side do not depend on $x$ and are finite. Indeed,  $\delta<1$ and the maps $p_0'(\cdot;\gamma)$ and $(g')^{-1}$ are integrable at $\infty $ because $\alpha>1+\delta$. Therefore, $v_0(\infty-;\gamma)<\infty$ and the proof is complete.
\end{proof}

 \appendix
 \section{Proofs of some technical results}\label{Appendix_tech}
 In this appendix, we collect the proofs of  Lemma \ref{lemma_diff_c}, Corollary \ref{fineteness_c} and Lemma \ref{le_dip_betagamma} in Section \ref{main}.
 \vspace{2mm}
 \begin{myproof}{Lemma \ref{lemma_diff_c}}
 	Fix $\beta,\,\gamma\in \mathbb{R}$ and $c_1,\,c_2\in I$ such that $c_1<c_2$.
 	All the equalities in \eqref{p_dep_c} and \eqref{v_dep_c} directly follow from the definitions of $p_c(\cdot;\beta,\gamma)$, see \eqref{scale_func_def}, and $v_c(\cdot;\beta,\gamma)$, see  \eqref{v_def}-\eqref{v_def2}. Here, we only show how to obtain the second equation in \eqref{v_dep_c} as an example. 	To do this, we take $x\in (c_2,r)$ and compute, according to \eqref{scale_func_def}-\eqref{v_def2}, using also \eqref{def_bc} and omitting $\gamma$ from the notation,
 	\begin{align*}
 		\notag v_{c_1}(x)&=v_{c_1}(c_2)+2\int_{c_2}^{x}\exp\bigg\{-2\int_{c_1}^{y}\frac{\tilde{b}(r)+\frac{K'(0)}{K(0)}\gamma}{\tilde{\sigma}^2(r)}\dd r\bigg\}\bigg(
 		\int_{c_1}^{y}\exp\bigg\{2\int_{c_1}^{z}\frac{\tilde{b}(r)+\frac{K'(0)}{K(0)}\gamma}{\tilde{\sigma}^2(r)}\dd r\bigg\}\frac{1}{\tilde{\sigma}^2(z)}\dd z\bigg)\dd y
 		\\
 		&=v_{c_1}(c_2)
 		+
 		2p'_{c_1}(c_2)\int_{c_2}^{x}p'_{c_2}(y)
 		\left(\int_{c_1}^{c_2}
 		\frac{1}{p'_{c_1}(z)\tilde{\sigma}^2(z)}\dd z
 		+
 		\frac{1}{p'_{c_1}(c_2)}
 		\int_{c_2}^{y}\frac{1}{p'_{c_2}(z)\tilde{\sigma}^2(z)}\dd z
 		\right)\dd y\notag
 		\\&=
 		v_{c_1}(c_2)
 		+
 		2p'_{c_1}(c_2)p_{c_2}(x)
 		\int_{c_1}^{c_2}
 		\frac{1}{p'_{c_1}(z)\tilde{\sigma}^2(z)}\dd z
 		+
 		v_{c_2}(x)
 		=
 		v_{c_1}(c_2)
 		+
 		p_{c_2}(x)
 		v'_{c_1}(c_2)
 		+
 		v_{c_2}(x).
 	\end{align*}
 	This is the  expression for $v_{c_1}$ in $(c_2,r)$ that we are looking for.
 \end{myproof}
\vspace{2mm}
\begin{myproof}{Corollary \ref{fineteness_c}}
	Consider $\beta,\,\gamma\in \mathbb{R}$ and $c_1,\,c_2\in I$ such that $c_1<c_2$. We first focus on demonstrating that $v_{c_2}(r-;\gamma)=\infty$ is equivalent to  $v_{c_1}(r-;\gamma)=\infty$. The implication 
	\[
	v_{c_2}(r-;\gamma)=\infty\quad \Longrightarrow \quad v_{c_1}(r-;\gamma)=\infty
	\] 
	follows directly from the second equality in \eqref{v_dep_c}, which gives $v_{c_1}(\cdot;\gamma)\ge v_{c_2}(\cdot;\gamma)$ in $(c_2,r)$.  For the reverse one, namely 
	\[
	v_{c_1}(r-;\gamma)=\infty\quad \Longrightarrow \quad v_{c_2}(r-;\gamma)=\infty,
	\]
	we suppose that $v_{c_1}(r-;\gamma)=\infty$ and argue by cases, corresponding to $p_{c_2}(r-;\gamma)<\infty$ or  $p_{c_2}(r-;\gamma)=\infty$. In the former, i.e., $p_{c_2}(r-;\gamma)<\infty$, we rewrite  the second equation in \eqref{v_dep_c} as
	\[
	v_{c_2}(x;\gamma)=v_{c_1}(x;\gamma)-v_{c_1}(c_2;\gamma)	-
	v'_{c_1}(c_2;\gamma)		p_{c_2}(x;\gamma),\quad x \in (c_2,r).
	\] 
	Taking the limit as $x\uparrow r$ we obtain $v_{c_2}(r-;\gamma)=\infty$. In the latter case, i.e., $p_{c_2}(r-;\gamma)=\infty$, the fact that $v_{c_2}(r-;\gamma)$ is infinite is independent from the hypothesis $v_{c_1}(r-;\gamma)=\infty$. Indeed, fixing a generic $\bar{c}\in(c_2,r)$, by \eqref{v_def2} we compute, for $x\in (\bar{c},r),$
	\begin{align*}
		v_{c_2}(x;\gamma)&=v_{c_2}(\bar{c};\gamma)+2\int_{\bar{c}}^{x}p'_{c_2}(y;\gamma)\left(\int_{c_2}^{y}\frac{\dd z}{p'_{c_2}(z;\gamma)\tilde{\sigma}^2(z)}\right)\dd y
		\\&\ge2
		\left(\int_{c_2}^{\bar{c}}\frac{\dd z}{p'_{c_2}(z;\gamma)\tilde{\sigma}^2(z)}\right)\left(p_{c_2}(x;\gamma)-p_{c_2}(\bar{c};\gamma)\right)
		.
	\end{align*}
	Thus, we deduce that $v_{c_2}(r-;\gamma)=\infty$ by letting $x\uparrow r$.
	
	The equivalence $$v_{c_1}(l+;\beta)=\infty\quad \Longleftrightarrow\quad  v_{c_2}(l+;\beta)=\infty$$ can be argued analogously, using the second equality in \eqref{p_dep_c} (which entails that $v_{c_2}(\cdot;\beta)\ge v_{c_1}(\cdot;\beta)$ in $(l,c_1)$) and showing the implication $p_{c_1}(l+;\beta)=-\infty\Longrightarrow v_{c_1}(l+;\beta)=\infty$.
	Therefore the proof is complete.
\end{myproof}
\vspace{2mm}
\begin{myproof}{Lemma \ref{le_dip_betagamma}}
	Fix $c\in I$ and parameters $\beta_i,\,\gamma_i,\,i=1,2$, according to the statement of the lemma, namely, $\beta_1\ge \beta_2$ and $\gamma_1\le \gamma_2$. 
	Differentiating \eqref{scale_func_def}, employing that $K'(0)\le 0$ and recalling \eqref{def_bc}, we obtain
	\begin{align*}
		\frac{p'_c(y;\gamma_1)}{p'_c(z;\gamma_1)}&=\exp\bigg\{-2\int_{z}^{y}\frac{\tilde{b}(h)}{\tilde{\sigma}^2(h)}\dd h-2\gamma_1\frac{K'(0)}{K(0)}\int_{z}^{y}\frac{1}{\tilde{\sigma}^2(h)}\dd h\bigg\}\\
		&\le 
		\exp\bigg\{-2\int_{z}^{y}\frac{\tilde{b}(h)}{\tilde{\sigma}^2(h)}\dd h-2\gamma_2\frac{K'(0)}{K(0)}\int_{z}^{y}\frac{1}{\tilde{\sigma}^2(h)}\dd h\bigg\}
		=\frac{p'_c(y;\gamma_2)}{p'_c(z;\gamma_2)}
		,\quad c<z<y<r.
	\end{align*}
	Hence, by \eqref{v_def}-\eqref{v_def2}, $v_c(\cdot; \gamma_1)\le v_c(\cdot; \gamma_2)$ in $(c,r)$. Analogously, we demonstrate that the  inequality $v_c(\cdot; \beta_1)\le v_c(\cdot; \beta_2)$ holds in $(l,c)$.
	
	Suppose now that $\int_{l}^{r}{\tilde{\sigma}^{-2}(z)}\dd z<\infty$. Arguing as in the above computation yields, for all  $y,\,z\in I$ such that $l<y<z<c$ or $c<z<y<r$,
	\begin{align*}
		\frac{p'_c(y;\beta_2,\gamma_2)}{p'_c(z;\beta_2,\gamma_2)}
		\le 
		\exp\bigg\{2\max\{\beta_1-\beta_2,\gamma_2-\gamma_1\}\frac{|K'(0)|}{K(0)}\int_{l}^{r}\frac{1}{\tilde{\sigma}^2(h)}\dd h\bigg\}
		\frac{p'_c(y;\beta_1,\gamma_1)}{p'_c(z;\beta_1,\gamma_1)}
		\eqqcolon C \frac{p'_c(y;\beta_1,\gamma_1)}{p'_c(z;\beta_1,\gamma_1)}
		,
	\end{align*}
	whence the second conclusion of the lemma.  This completes the proof.
\end{myproof}

\end{document}